\documentclass[10pt,reqno]{amsart}
\usepackage{amsmath}
\usepackage{amssymb}
\usepackage{dsfont}
\usepackage[dvips,draft,final]{graphics}
\usepackage{amssymb,amsmath}
\usepackage[T1]{fontenc}
\usepackage{fancyhdr}
\usepackage{url}

\usepackage{color}
\usepackage{graphicx}

\def\squarebox#1{\hbox to #1{\hfill\vbox to #1{\vfill}}}

\newtheorem{Thm}{Theorem}[section]

\newtheorem{lem}{Lemma}[section]

\numberwithin{equation}{section}

\newcommand{\bel}{\begin{equation} \label}
\newcommand{\ee}{\end{equation}}
\newcommand{\re}{\mathfrak R}
\newcommand{\pd}{\partial}

\newcommand{\R}{\mathbb{R}}

\def\epsilon{\varepsilon}
\def\phi {\varphi}

\newtheorem{rem}{Remark}[section]
\newtheorem{prop}{Proposition}[section]

\providecommand{\abs}[1]{\left\lvert#1\right\rvert}
\providecommand{\norm}[1]{\left\lVert#1\right\rVert}
\numberwithin{equation}{section}

\renewcommand{\leq}{\leqslant}
\renewcommand{\geq}{\geqslant}
\providecommand{\abs}[1]{\left\lvert#1\right\rvert}
\providecommand{\norm}[1]{\left\lVert#1\right\rVert}

\def\beq{\begin{equation}}
\def\eeq{\end{equation}}
\newcommand{\bea}{\begin{eqnarray}}
\newcommand{\eea}{\end{eqnarray}}
\newcommand{\beas}{\begin{eqnarray*}}
\newcommand{\eeas}{\end{eqnarray*}}

{

\begin{document}

\title[Determination of  singular time-dependent  coefficients for  wave equations]{Determination of  singular time-dependent  coefficients for  wave equations from full and partial data}

\author{Guanghui Hu}
\address{Beijing Computational Science Research Center, Building 9, East Zone, ZPark II, No.10 Xibeiwang East Road, Haidian District, Beijing 100193, China.}
\email{hu@csrc.ac.cn}
\author{Yavar Kian}
\address{Aix Marseille Univ, Universit\'e de Toulon, CNRS, CPT, Marseille, France.}
\email{yavar.kian@univ-amu.fr}

\begin{abstract}
We study   the problem of  determining  uniquely  a time-dependent singular potential $q$, appearing in   the wave equation $\partial_t^2u-\Delta_x u+q(t,x)u=0$ in $Q=(0,T)\times\Omega$ with $T>0$ and $\Omega$  a  $ \mathcal C^2$ bounded domain of $\R^n$, $n\geq2$. We start by considering the unique determination of some singular time-dependent coefficients
from  observations on $\partial Q$. Then, by weakening the singularities of the set of admissible coefficients, we manage to reduce the set of data that still guaranties unique recovery of such a coefficient. To our best knowledge, this paper is the first claiming unique determination of unbounded time-dependent coefficients, which is motivated by the problem of determining  general nonlinear terms appearing in nonlinear wave equations.

\end{abstract}


\maketitle

\section{Introduction}
\label{sec-intro}
\setcounter{equation}{0}
\subsection{Statement of the problem }
Let  $\Omega$   be a $\mathcal C^2$ bounded domain  of $\R^n$, $n\geq2$, and  fix  $\Sigma=(0,T)\times\partial\Omega$, $Q=(0,T)\times\Omega$ with $0<T<\infty$.  We consider the wave equation \begin{equation}\label{wave}\partial_t^2u-\Delta_x u+q(t,x)u=0,\quad (t,x)\in Q,\end{equation}
where   the potential $q$ is assumed to be an unbounded real valued coefficient.  In this paper we
seek unique determination of  $q$ from observations  of  solutions of \eqref{wave} on $\partial Q$.
\subsection{Obstruction to uniqueness and set of full data for our problem}
Let $\nu$ be the outward unit normal vector to $\partial\Omega$, $\partial_\nu=\nu\cdot\nabla_x$ the normal derivative and from now on let $\Box$ be  the differential operators $\Box:=\partial_t^2-\Delta_x$. It has been proved by \cite{RS1},  that, for $T>\textrm{Diam}(\Omega)$, the data\begin{equation}\label{data1} \mathcal A_{q}=\{(u_{|\Sigma},\partial_\nu u_{|\Sigma}):\ u\in L^2(Q),\ \Box u+qu=0,\ u_{|t=0}=\partial_tu_{|t=0}=0\}\end{equation}
determines uniquely a time-independent potential $q$.   On the other hand,   due to domain of dependence arguments, there is  no hope to recover  even smooth time-dependent coefficients restricted  to the set $$D=\{(t,x)\in Q:\  t\in(0,\textrm{Diam}(\Omega)/2)\cup (T-\textrm{Diam}(\Omega)/2,T),\ \textrm{dist}(x,\partial\Omega)> \min(t,T-t)\}$$ from the data $\mathcal A_{q}$ (see \cite[Subsection 1.1]{Ki2}). Therefore, even  when $T$ is large, for the global recovery  of general time-dependent coefficients the information on the bottom $t=0$  and the top $t=T$ of $Q$ are unavoidable.
 Thus, for our problem the  extra information on $\{t=0\}$ and  $\{t=T\}$, of solutions $u$ of \eqref{wave}, can not be completely removed. In this context, we introduce the set of data
\[C_{q}=\{(u_{|\Sigma},u_{t=0},\partial_t u_{|t=0},\partial_\nu u_{|\Sigma}, u_{|t=T},\partial_tu_{|t=T}):\ u\in L^2(Q),\ \Box u+qu=0 \}.\]
and we recall that \cite{I} proved that, for $q\in L^\infty(Q)$, the data $C_q$ determines uniquely $q$. From now on we will refer to $C_q$ as the set of full data for our problem and we mention that \cite{Ki3,Ki2,Ki4} proved recovery of bounded time-dependent coefficients $q$ from partial data corresponding to partial knowledge of the set $C_q$. The goal of the present paper is to prove recovery of singular  time-dependent coefficients $q$  from full and partial data.

\subsection{Physical and mathematical motivations }

Physically speaking, our inverse problem consists of determining   unstable properties such as some rough time evolving  density of an inhomogeneous medium from disturbances generated on  the boundary and at  initial time, and measurements of the response. The goal is to determine the function $q$ which describes  the property of the medium. Moreover, singular time-dependent coefficients can be associated to some unstable time-evolving phenomenon that can not be modeled  by bounded  time-dependent coefficients or time independent coefficients.

Let us also observe that, according to \cite{CK2,I2}, for parabolic equations the recovery
of   nonlinear terms, appearing in some suitable nonlinear  equations,  can be reduced to the determination of  time-dependent coefficients. In this context, the information that allows to recover the nonlinear term is transferred, throw a linearization process,  to a time-dependent coefficient depending explicitly on some solutions of the nonlinear problem. In contrast to parabolic equations, due to the weak regularity of solutions, it is not clear that this process allows to transfer the recovery of nonlinear terms, appearing in a nonlinear wave equation, to a bounded time-dependent coefficient. Thus, in order to expect an application of the strategy set by \cite{CK2,I2} to  the recovery of nonlinear terms for nonlinear wave equations, it seems important to consider recovery of singular time-dependent coefficients.

\subsection{known results }

The problem of determining  coefficients appearing in hyperbolic equations   has  attracted many attention over the last decades.  This problem has been stated in terms of recovery of a time-independent potential $q$ from the set $\mathcal A_{q}$. For instance,  \cite{RS1} proved that  $\mathcal A_{q}$ determines uniquely a time-independent  potential $q$, while \cite{E1} proved that  partial boundary observations are sufficient for this problem.   We recall also that \cite{BD,BJY,Ki,SU2} studied the  stability issue for this problem.

 Several authors considered also the problem of determining time-dependent coefficients appearing in wave equations. In \cite{St}, the authors shown that the knowledge of scattering data determines uniquely a smooth   time-dependent potential.   In \cite{RS}, the authors studied  the recovery of a time-dependent potential $q$ from  data on the boundary $\pd\Omega$ for all time given by  $(u_{|\R\times\partial\Omega}, \partial_\nu u_{|\R\times\partial\Omega})$ of forward solutions of \eqref{wave}  on the infinite time-space cylindrical domain $\R_t\times\Omega$  instead of $Q$. As for \cite{RR}, the authors  considered this  problem at finite time on $Q$ and they proved the recovery of $q$ restricted to some strict subset of $Q$ from $\mathcal A_{q}$.  Isakov established in \cite[Theorem 4.2]{I} unique global determination of  general time-dependent potentials on the whole domain $Q$ from the important set of full data $C_{q}$. By applying  a result of unique continuation for wave equation, which is valid only for coefficients analytic with respect to the time variable (see for instance the counterexample of \cite{AlBa}), \cite{E2} proved unique recovery of time-dependent coefficients from  partial knowledge of the data $\mathcal A_{q}$.  In \cite{S}, the author extended the result of \cite{RS}. Moreover, \cite{W} established the stable   recovery of  X-ray transforms of time-dependent potentials and \cite{BB,A} proved log-type stability in the determination of time-dependent coefficients with data similar to  \cite{I} and \cite{RR}.  In \cite{Ki3,Ki2,Ki4}, the author proved   uniqueness and stability in the recovery of several time-dependent coefficients  from partial knowledge of the full set of data $C_q$. It seems that the results of \cite{Ki3,Ki2,Ki4} are stated with the weakest conditions so far that allows to recover general bounded time-dependent coefficients. More recently, \cite{KiOk} proved unique determination of such coefficients on Riemannian manifolds. We mention also the work of \cite{SY} who determined some information about time-dependent coefficients  from the Dirichlet-to-Neumann map  on a cylinder-like Lorentzian manifold related
to the wave equation.
We refer to the work \cite{CK,CKS,FK,GK,KS} for determination of time-dependent coefficients for fractional diffusion, parabolic and Schr\"odinger equations have been considered.

In all the above mentioned results, the authors considered time-dependent coefficients that are at least bounded. There have been several  works dealing with recovery of non-smooth coefficients appearing in elliptic equations such as \cite{CR,CT,FKS,HT}. Nevertheless, to our best knowledge, except the present paper, there is no work in the mathematical literature dealing with the recovery of singular time-dependent coefficients $q$ even from the important set of full data $C_q$.

\subsection{Main results}

 The main purpose of this paper is to prove  the unique global  determination of  time-dependent   and unbounded coefficient  $q$ from partial knowledge of the observation of solutions on $\partial Q=(\{0\}\times\overline{\Omega})\cup\Sigma\cup(\{T\}\times\overline{\Omega})$. More precisely, we would like to prove unique recovery of unbounded coefficient $q\in L^{p_1}(0,T;L^{p_2}(\Omega))$, $p_1\geq 2$, $p_2\geq n$, from partial knowledge of the full set of data $C_q$. We start by considering the recovery of some general unbounded coefficient $q$ from restriction of $C_q$ on the bottom  $t=0$ and top  $t=T$ of the time-space cylindrical domain $Q$. More precisely, for $q\in L^{p_1}(0,T;L^{p_2}(\Omega))$, $p_1\geq 2$, $p_2\geq n$, we consider the recovery of $q$ from the set of data
\[C_{q}(0)=\{(u_{|\Sigma},\partial_t u_{|t=0},\partial_\nu u, u_{|t=T},\partial_tu_{|t=T}):\ u\in \mathcal K(Q),\ \Box u+qu=0,\ u_{|t=0}=0 \},\]
or the set of data
\[C_{q}(T)=\{(u_{|\Sigma},u_{t=0},\partial_t u_{|t=0},\partial_\nu u, u_{|t=T}):\ u\in \mathcal K(Q),\ \Box u+qu=0 \},\]
where $\mathcal K(Q)=\mathcal C([0,T];H^1(\Omega))\cap \mathcal C^1([0,T];L^2(\Omega))$. In addition, assuming that $T>\textrm{Diam}(\Omega)$, we prove the recovery of $q$ from the set of data
\[C_{q}(0,T)=\{(u_{|\Sigma},\partial_t u_{|t=0},\partial_\nu u, u_{|t=T}):\ u\in \mathcal K(Q),\ \Box u+qu=0,\ u_{|t=0}=0 \}.\]
Our first main result can be stated as follows
\begin{Thm}\label{thm3}
Let $p_1\in(2,+\infty)$, $p_2\in(n,+\infty)$ and let $q_1,\ q_2 \in L^{p_1}(0,T;L^{p_2}(\Omega))$.
Then, either of the following conditions:
\begin{equation}\label{thm3a}C_{q_1}(0)=C_{q_2}(0),\end{equation}
\begin{equation}\label{thm3b}C_{q_1}(T)=C_{q_2}(T),\end{equation}
implies that  $q_1=q_2$. Moreover, assuming that $T>\textrm{Diam}(\Omega)$, the condition
\begin{equation}\label{thm3c}C_{q_1}(0,T)=C_{q_2}(0,T)\end{equation}
implies that $q_1=q_2$.
\end{Thm}

We  consider also the recovery of a time-dependent   and unbounded coefficient  $q$ from restriction of the data $C_{q}$ on the  lateral boundary $\Sigma$. Namely, for all $\omega\in\mathbb S^{n-1}=\{x\in\R^n:\ \abs{x}=1\}$ we introduce the $\omega$-shadowed and $\omega$-illuminated faces
\[\partial\Omega_{+,\omega}=\{x\in\partial\Omega:\ \nu(x)\cdot\omega>0\},\quad \partial\Omega_{-,\omega}=\{x\in\partial\Omega:\ \nu(x)\cdot\omega\leq0\}\]
of $\partial\Omega$. Here, for all $k\in\mathbb N^*$, $\cdot$ denotes the scalar product in $\R^k$ given by
\[ x\cdot y=x_1y_1+\ldots +x_ky_k,\quad x=(x_1,\ldots,x_k)\in \R^k,\ y=(y_1,\ldots,y_k)\in \R^k.\]
We define also the parts of the lateral boundary $\Sigma$ taking the form
$\Sigma_{\pm,\omega}=(0,T)\times \partial\Omega_{\pm,\omega}$.
We fix $\omega_0\in \mathbb S^{n-1}$ and we consider $V=(0,T)\times V'$  with $V'$  a closed  neighborhood of  $\partial\Omega_{-,\omega_0}$ in $\partial\Omega$.
Then, we study the recovery of $q\in L^{p}(Q)$, $p>n+1$,  from the data
\[C_{q}(T,V)=\{(u_{|\Sigma},u_{|t=0},\partial_t u_{|t=0},\partial_\nu u_{|V}, u_{|t=T}):\ u\in H^1(Q),\ \Box u+qu=0\}\]
and the determination of a time-dependent   coefficient   $q\in L^{\infty}(0,T;L^p(\Omega))$, $p>n$,  from the data
\[C_{q}(0,T,V)=\{(u_{|\Sigma},\partial_t u_{|t=0},\partial_\nu u_{|V}, u_{|t=T}):\ u\in L^2(0,T;H^1(\Omega)),\ \Box u+qu=0,\ u_{|t=0}=0\}.\]
We refer to Section 2 for the definition of this set. Our main result can be  stated as follows.

\begin{Thm}\label{thm1}
Let $p\in(n+1,+\infty)$ and let $q_1,\ q_2 \in L^p(Q)$.
Then, the condition
\begin{equation}\label{thm1a}C_{q_1}(T,V)=C_{q_2}(T,V)\end{equation}
implies that  $q_1=q_2$.
\end{Thm}

\begin{Thm}\label{thm2}
Let $p\in(n,+\infty)$ and let $q_1,\ q_2 \in L^{\infty}(0,T;L^p(\Omega))$.
Then, the condition
\begin{equation}\label{thm2a}C_{q_1}(0,T,V)=C_{q_2}(0,T,V)\end{equation}
implies that  $q_1=q_2$.
\end{Thm}

To our best knowledge the result of Theorem \ref{thm3}, \ref{thm1} and \ref{thm2} are the first results  claiming unique determination of unbounded time-dependent coefficients for the wave equation. In Theorem \ref{thm3}, we prove recovery of coefficients $q$, that can admit some singularities, by making restriction on the set of full data $C_q$ on the bottom $t=0$ and the top $t=T$ of $Q$. While, in Theorem  \ref{thm1} and \ref{thm2}, we consider less singular  time-dependent  coefficients, in order to restrict the data on the lateral boundary $\Sigma=(0,T)\times\pd\Omega$.

We mention also that the uniqueness result of Thorem \ref{thm2} is stated with data close to the one considered by \cite{Ki3,Ki2}, where determination  of bounded time-dependent potentials is proved with conditions that seems to be one of the weakest so far. More precisely,  the only difference between \cite{Ki3,Ki2} and Theorem \ref{thm2} comes from the restriction on the Dirichlet boundary condition  (\cite{Ki3,Ki2} consider Dirichlet boundary condition supported on a neighborhood of the $\omega_0$-shadowed face, while in Theorem \ref{thm2}  we do not restrict the support of the Dirichlet boundary).

In the present paper we consider two different approaches which depend mainly on the restriction that we make on the set of full data $C_q$. For Theorem \ref{thm3},
we use geometric optics solutions corresponding to  oscillating solutions of the form
\bel{osci}u(t,x)=\sum_{j=1}^Na_j(t,x)e^{i\lambda \psi_j(t,x)}+R_\lambda(t,x),\quad (t,x)\in Q,\ee
with $\lambda>1 $ a large parameter, $R_\lambda$ a remainder term that admits a decay with respect to the parameter $\lambda$ and $\psi_j$, $j=1,..,N$,  real valued. For $N=1$, these solutions correspond to a classical tool for proving determination of time independent or time-dependent coefficients (e. g. \cite{BB,BCY,BD,A,RR,RS,RS1}). In a similar way to \cite{KiOk}, we consider in Theorem \ref{thm3}
solutions of the form \eqref{osci} with $N=2$ in order to be able to restrict the data at $t=0$ and $t=T$ while  avoiding a "reflection". It seems that in the approach set so far for the construction of  solutions of the form \eqref{osci}, the decay of the remainder term $R_\lambda$ relies in an important way to the fact that the coefficient $q$ is bounded (or time independent). In this paper, we prove how this construction can be extended to unbounded time-dependent coefficients.

The approach used for Theorem \ref{thm3} allows in a quite straightforward way to restrict the data on the bottom $t=0$ and on the top $t=T$ of $Q$. Nevertheless, it is not clear how one can extend this approach to restriction on the lateral boundary $\Sigma$ without requiring additional smoothness or geometrical assumptions. For this reason, in order to consider restriction on $\Sigma$, we use a different approach where the oscillating solutions \eqref{osci} are replaced by exponentially growing and exponentially decaying solutions of the form
\bel{exp}u(t,x)=e^{\pm \lambda (t+x\cdot\omega)}(a(t,x)+w_\lambda(t,x)),\quad (t,x)\in Q,\ee
where $\omega\in\mathbb S^{n-1}$ and $w_\lambda$ admits a decay with respect to the parameter $\lambda$. The idea of this approach, which is inspired by \cite{BJY,Ki3,Ki2,Ki4} (see also \cite{BU,KSU} for elliptic equations), consists of combining results of density of products of solutions with  Carleman estimates with linear weight in order to be able to restrict at the same time the data  on the bottom $t=0$, on the top $t=T$ and on the lateral boundary $\Sigma$ of $Q$. For the construction of these solutions, we use Carleman estimates in negative order Sobolev space. To our best knowledge this is the first extension of this approach to singular time-dependent coefficients.

\subsection{Outline}

This paper is organized as follows. In Section 2, we start with some preliminary results and we define the set of data $C_q(0)$, $C_q(T)$, $C_q(0,T)$, $C_q(T,V)$ and $C_q(0,T,V)$. In Section 3, we prove Theorem \ref{thm3} by mean of  geometric optics solutions of the form \eqref{osci}. Then, Section 4 and Section 5 are respectively devoted to the proof of Theorem \ref{thm1} and Theorem \ref{thm2}.

\section{Preliminary results}
In the present section we define the set of data $C_{q}(T,V)$, $C_{q}(0,T,V)$ and we recall some properties of the solutions of \eqref{wave} for any $q\in L^{p_1}(Q)$, with $p_1>n+1$, or, for $q\in L^\infty (0,T;L^{p_2}(\Omega))$, with $p_2>n$.  For this purpose, in a similar way to  \cite{Ki2}, we will introduce some preliminary tools. We define the  space
\[H_{\Box}(Q)=\{u\in H^1(Q):\ \Box u=(\partial_t^2-\Delta_x) u\in L^2(Q)\},\]\[ H_{\Box,*}(Q)=\{u\in L^2(0,T;H^1(\Omega)):\ \Box u=(\partial_t^2-\Delta_x) u\in L^2(Q)\},\]
 with the norm
\[\norm{u}^2_{H_{\Box}(Q)}=\norm{u}_{H^1(Q)}^2+\norm{(\partial_t^2-\Delta_x) u}_{L^2(Q)}^2,\]
\[\norm{u}^2_{H_{\Box,*}(Q)}=\norm{u}_{L^2(0,T;H^1(\Omega))}^2+\norm{(\partial_t^2-\Delta_x) u}_{L^2(Q)}^2.\]
We consider also  the  space
\[S=\{u\in H^1(Q):\ (\partial_t^2-\Delta_x)u=0\} \ (\textrm{resp } S_*=\{u\in L^2(0,T;H^1(\Omega)):\ (\partial_t^2-\Delta_x)u=0\})\]
and topologize it as a closed subset of $H^1(Q)$ (resp $L^2(0,T;H^1(\Omega))$).  In view of  \cite[Proposition 4]{Ki2},  the maps
\[\tau_0w=(w_{\vert\Sigma},w_{\vert t=0},\partial_t w_{\vert t=0}) ,\quad \tau_1w=(\partial_\nu w_{\vert\Sigma},w_{\vert t=T},\partial_t w_{\vert t=T}), \quad w\in \mathcal C^\infty(\overline{Q}),\]
can be extended continuously to $\tau_0:H_{\Box,*}(Q)\rightarrow H^{-3}(0,T; H^{-\frac{1}{2}}(\partial\Omega))\times H^{-2}(\Omega)\times H^{-4}(\Omega)$, $\tau_1:H_{\Box,*}(Q)\rightarrow H^{-3}(0,T; H^{-\frac{3}{2}}(\partial\Omega))\times H^{-2}(\Omega)\times H^{-4}(\Omega)$. Here for all $ w\in \mathcal C^\infty(\overline{Q})$ we set
\[\tau_0w=(\tau_{0,1}w,\tau_{0,2}w,\tau_{0,3}w),\quad \tau_1w=(\tau_{1,1}w,\tau_{1,2}w,\tau_{1,3}w),\]
where
\[ \tau_{0,1}w=w_{\vert\Sigma},\    \tau_{0,2}w=w_{\vert t=0},\ \tau_{0,3}w=\partial_tw_{\vert t=0},\  \tau_{1,1}w=\partial_\nu w_{\vert\Sigma},\ \tau_{1,2}w=w_{\vert t=T},\ \tau_{1,3}w=\partial_tw_{\vert t=T}.\]
Therefore, we can introduce
\[\mathcal H=\{\tau_0u:\ u\in H_{\Box}(Q)\}\subset H^{-3}(0,T; H^{-\frac{1}{2}}(\partial\Omega))\times H^{-2}(\Omega)\times H^{-4}(\Omega),\]
$$\mathcal H_*=\{(\tau_{0,1}u,\tau_{0,3}u):\ u\in H_{\Box,*}(Q),\ \tau_{0,2}u=0\}\subset  H^{-3}(0,T; H^{-\frac{1}{2}}(\partial\Omega))\times H^{-4}(\Omega).$$
By repeating the arguments used in \cite[Proposition 1]{Ki2}, one can check that  the restriction of $\tau_0$ to $S$ (resp $S_*$) is one to one and onto. Thus, we can   use  $({\tau_0}_{|S})^{-1}$ (resp $({\tau_0}_{|S_*})^{-1}$) to define the norm of $\mathcal H$ (resp $\mathcal H_*$) by
\[\norm{(f,v_0,v_1)}_{\mathcal H}=\norm{({\tau_0}_{|S})^{-1}(f,v_0,v_1)}_{H^1(Q)},\quad (f,v_0,v_1)\in\mathcal H,\]
$$(\textrm{resp} \norm{(f,v_1)}_{\mathcal H_*}=\norm{({\tau_0}_{|S_*})^{-1}(f,0,v_1)}_{L^2(0,T;H^1(\Omega))},\quad (f,v_1)\in\mathcal H_*).$$
 Let us consider the initial boundary value problem (IBVP in short)
\bel{eq2}\left\{ \begin{array}{rcll} \partial_t^2v-\Delta_x v+qv =  F(t,x),&  (t,x) \in Q ,\\

v(0,x)=v_0(x),\ \ \partial_t v(0,x)=v_1(x),& x\in\Omega\\
 v(t,x)=0,  & (t,x)\in\Sigma.& \end{array}\right.
\ee

We have the following well-posedness result for this IBVP when $q$ is unbounded.
\begin{prop}\label{p1} Let $p_1\in(1,+\infty)$ and $p_2\in(n,+\infty)$. For $q\in L^{p_1}(0,T;L^{p_2}(\Omega))$, $v_0\in H^1_0(\Omega)$, $v_1\in L^2(\Omega)$ and $F\in L^2(Q)$, problem \eqref{eq2} admits a unique solution $v\in \mathcal C([0,T];H^1_0(\Omega))\cap \mathcal C^1([0,T];L^2(\Omega))$ satisfying
\bel{p1a} \norm{v}_{\mathcal C([0,T];H^1_0(\Omega))}+\norm{v}_{\mathcal C^1([0,T];L^2(\Omega))}\leq C (\norm{v_0}_{H^1(\Omega)}+\norm{v_1}_{L^2(\Omega)}+\norm{F}_{L^2(Q)}),\ee
with $C$ depending only on $p_1$, $p_2$, $n$, $T$, $\Omega$ and any $M\geq \norm{q}_{L^{p_1}(0,T;L^{p_2}(\Omega))}$.
\end{prop}
\begin{proof}

According to the second part of the proof of \cite[Theorem 8.1, Chapter 3]{LM1}, \cite[Remark 8.2, Chapter 3]{LM1} and \cite[Theorem 8.3, Chapter 3]{LM1}, the proof of this proposition will be completed if we show that for any $v\in W^{2,\infty}(0,T;H^1_0(\Omega))$ solving \eqref{eq2} the a priori estimate \eqref{p1a} holds true. Without lost of generality we assume that $v$ is real valued. From now on we consider this estimate. We define the enery $E(t)$ at time $t\in[0,T]$ by
$$E(t):=\int_\Omega \left( |\partial_t v(t,x)|^2+|\nabla_x v(t,x)|^2\right)dx.$$
Multiplying \eqref{eq2} by $\partial_tv$ and integrating by parts we get
\bel{p1b} E(t)-E(0)=-2\int_0^t\int_\Omega q(s,x)v(s,x)\partial_t v(s,x)dxds+2\int_0^t\int_\Omega F(s,x)\partial_t v(s,x)dxds.\ee
On the other hand, we have
\bel{p1c}\abs{\int_0^t\int_\Omega q(s,x)v(s,x)\partial_t v(s,x)dxds}\leq \int_0^t\norm{qv(s,\cdot)}_{L^2(\Omega)}\norm{\partial_tv(s,\cdot)}_{L^2(\Omega)}ds.\ee
Applying the Sobolev embedding theorem and the H\"older inequality, for all $s\in(0,T)$ we get
$$\begin{aligned}\norm{qv(s,\cdot)}_{L^2(\Omega)}&\leq \norm{q(s,\cdot)}_{L^{p_2}(\Omega)}\norm{v(s,\cdot)}_{L^{\frac{2p_2}{p_2-2}}(\Omega)}\\
\ &\leq C\norm{q(s,\cdot)}_{L^{p_2}(\Omega)}\norm{v(s,\cdot)}_{H^{\frac{n}{p_2}}(\Omega)}\\
\ &\leq C\norm{q(s,\cdot)}_{L^{p_2}(\Omega)}\norm{v(s,\cdot)}_{H^1(\Omega)}\end{aligned}$$
with $C$ depending only on $\Omega$. Then, the Poincarr\'e inequality implies
$$\norm{qv(s,\cdot)}_{L^2(\Omega)}\leq C\norm{q(s,\cdot)}_{L^{p_2}(\Omega)}\norm{\nabla_xv(s,\cdot)}_{L^2(\Omega)}\leq C\norm{q(s,\cdot)}_{L^{p_2}(\Omega)} E(s)^{\frac{1}{2}},$$
where  $C$ depends only on $\Omega$. Thus, from \eqref{p1c}, we get
\bel{p1d}\abs{\int_0^t\int_\Omega q(s,x)v(s,x)\partial_t v(s,x)dxds}\leq C\int_0^t\norm{q(s,\cdot)}_{L^{p_2}(\Omega)}E(s)ds\leq \norm{q}_{L^{p_1}(0,T;L^{p_2}(\Omega))}\left(\int_0^tE(s)^{\frac{p_1}{p_1-1}}ds\right)^{\frac{p_1-1}{p_1}}.\ee
In the same way, an application of  the H\"older inequality yields
$$\begin{aligned}\abs{\int_0^t\int_\Omega F(s,x)\partial_t v(s,x)dxds}&\leq T^{\frac{1}{2p_1}}\norm{F}_{L^2(Q)}\left(\int_0^tE(s)^{\frac{p_1}{p_1-1}}ds\right)^{\frac{p_1-1}{2p_1}}\\
\ &\leq T^{\frac{1}{p_1}}\norm{F}_{L^2(Q)}^2+\left(\int_0^tE(s)^{\frac{p_1}{p_1-1}}\right)^{\frac{p_1-1}{p_1}}.\end{aligned}$$
Combining this estimate with \eqref{p1b}-\eqref{p1d}, we deduce that
$$E(t)\leq E(0)+ C\norm{F}_{L^2(Q)}^2+ C\left(\int_0^tE(s)^{\frac{p_1}{p_1-1}}\right)^{\frac{p_1-1}{p_1}},$$
where $C$ depends only on $T$, $\Omega$ and any $M\geq \norm{q}_{L^{p_1}(0,T;L^{p_2}(\Omega))}$. By taking the power $\frac{p_1}{p_1-1}$ on both side of this inequality, we get
$$E(t)^{\frac{p_1}{p_1-1}}\leq C\left(\norm{v_0}_{H^1(\Omega)}+\norm{v_1}_{L^2(\Omega)}+\norm{F}_{L^2(Q)}\right)^{\frac{2p_1}{p_1-1}}+C\int_0^tE(s)^{\frac{p_1}{p_1-1}}ds.$$
Then, the Gronwall inequality implies
$$\begin{aligned}E(t)^{\frac{p_1}{p_1-1}}&\leq C\left(\norm{v_0}_{H^1(\Omega)}+\norm{v_1}_{L^2(\Omega)}+\norm{F}_{L^2(Q)}\right)^{\frac{2p_1}{p_1-1}}e^{Ct}\\
\ &\leq C\left(\norm{v_0}_{H^1(\Omega)}+\norm{v_1}_{L^2(\Omega)}+\norm{F}_{L^2(Q)}\right)^{\frac{2p_1}{p_1-1}}e^{CT}.\end{aligned}$$
From this last estimate one can easily deduce \eqref{p1a}.

\end{proof}

Let us introduce  the IBVP
\begin{equation}\label{eq1}\left\{\begin{array}{ll}\partial_t^2u-\Delta_x u+q(t,x)u=0,\quad &\textrm{in}\ Q,\\  u(0,\cdot)=v_0,\quad \partial_tu(0,\cdot)=v_1,\quad &\textrm{in}\ \Omega,\\ u=g,\quad &\textrm{on}\ \Sigma.\end{array}\right.\end{equation}
We are now in position to state existence and uniqueness of solutions of this IBVP for $(g,v_0,v_1)\in \mathcal H$ and $q\in L^p(Q)$, $p>n+1$.
\begin{prop}\label{p2} Let $(g,v_0,v_1)\in \mathcal H$,  $q\in L^p(Q)$, $p>n+1$. Then, the IBVP \eqref{eq1} admits a unique weak solution $u\in H^1(Q)$ satisfying
\bel{p2a}
\norm{u}_{H^1(Q)}\leq C\norm{(g,v_0,v_1)}_{\mathcal H}
\ee
and the boundary operator $B_{q}: (g,v_0,v_1)\mapsto (\tau_{1,1}u_{|V}, \tau_{1,2}u)$ is a bounded operator from $\mathcal H$ to\\
 $H^{-3}(0,T; H^{-\frac{3}{2}}(V'))\times H^{-2}(\Omega)$.
\end{prop}
 \begin{proof} We split $u$ into two terms $u=v+\tau_0^{-1}(g,v_0,v_1)$ where $v$ solves
\bel{Eq1}\left\{ \begin{array}{rcll} \partial_t^2v-\Delta_x v+qv& = & q\tau_0^{-1}(g,v_0,v_1), & (t,x) \in Q ,\\

v_{\vert t=0}=\partial_t v_{\vert t=0}&=&0,&\\
 v_{\vert\Sigma}& = & 0.& \end{array}\right.
\ee
Since $q\in L^p(Q)$ and $\tau_0^{-1}(g,v_0,v_1)\in H^1(Q)$, by the Sobolev embedding theorem we have $q\tau_0^{-1}(g,v_0,v_1)\in L^2(Q)$. Thus, according to Proposition \ref{p1} one can check that the IBVP \eqref{Eq1} has a unique solution $v\in \mathcal C^1([0,T];L^2(\Omega))\cap \mathcal C([0,T];H^1_0(\Omega))$ satisfying
\bel{p2b}
\begin{aligned}\norm{v}_{\mathcal C^1([0,T];L^2(\Omega))}+\norm{v}_{\mathcal C([0,T];H^1_0(\Omega))} &\leq C\norm{q\tau_0^{-1}(g,v_0,v_1)}_{L^2(Q)}\\
\ &\leq C\norm{q}_{L^p(Q)}\norm{\tau_0^{-1}(g,v_0,v_1)}_{H^1(Q)}.\end{aligned}\ee
Thus, $u=v+\tau_0^{-1}(g,v_0,v_1)$ is the unique solution of \eqref{eq1} and estimate \eqref{p2b} implies \eqref{p2a}. Now let us consider the last part of the proposition. For this purpose, let $(g,v_0,v_1)\in\mathcal H$ and let $u\in H^1(Q)$ be the solution of \eqref{eq1}. Note first that  $(\partial_t^2-\Delta_x) u=-qu\in L^2(Q)$. Therefore,  $u\in H_{\Box}(Q)$ and  $\tau_{1,1}u\in H^{-3}(0,T; H^{-\frac{3}{2}}(\partial\Omega))$, $\tau_{1,2}u \in H^{-2}(\Omega)$ with
\[\begin{aligned}\norm{\tau_{1,1}u}^2+\norm{\tau_{1,2}u}^2\leq C^2\norm{u}^2_{H_{\Box}(Q)}&=C^2(\norm{u}^2_{H^1(Q))}+\norm{qu}^2_{L^2(Q)})\\
\ &\leq C^2(1+\norm{q}^2_{L^p(Q)})\norm{u}_{H^1(Q)}^2.\end{aligned}\]
Combining this with \eqref{p2a}, we find that $B_{q}$ is a bounded operator from $\mathcal H$ to $H^{-3}(0,T; H^{-\frac{3}{2}}(V'))\times H^{-2}(\Omega)$.\end{proof}

From now on, we define the set $C_{q}(T,V)$  by
\[C_{q}(T,V)=\{(g,v_0,v_1,B_{q}(g,v_0,v_1)):\ (g,v_0,v_1)\in \mathcal H\}.\]
In the same way, for $q\in L^{p_1}(0,T;L^{p_2}(\Omega))$, $p_1\geq2,p_2> n$, we consider the set $C_q(T)$, $C_{q}(0)$, $C_{q}(0,T)$ introduced before Theorem \ref{thm3}. Using similar arguments to Proposition \ref{p2} we can prove the following.

\begin{prop}\label{p3} Let $(g,v_1)\in \mathcal H_*$ with $v_0=0$ and let  $q\in L^\infty(0,T;L^p(\Omega))$, $p>n$. Then, the IBVP \eqref{eq1} admits a unique weak solution $u\in L^2(0,T;H^1(\Omega))$ satisfying
\bel{p3a}
\norm{u}_{L^2(0,T;H^1(\Omega))}\leq C\norm{(g,v_1)}_{\mathcal H_*}
\ee
and the boundary operator $B_{q,*}: (g,v_1)\mapsto (\tau_{1,1}u_{|V}, \tau_{1,2}u)$ is a bounded operator from $\mathcal H_*$ to\\
 $H^{-3}(0,T; H^{-\frac{3}{2}}(V'))\times H^{-2}(\Omega)$.
\end{prop}
We define the set $C_{q}(0,T,V)$ by
\[C_{q}(0,T,V)=\{(g,v_1,B_{q,*}(g,v_1)):\ (g,v_1)\in \mathcal H_*\}.\]

\section{Proof of Theorem \ref{thm3}}

The goal of this section is to prove Theorem \ref{thm3}. For this purpose,  we consider special solutions $u_j$ of the equation \bel{wave1}\partial_t^2u_j-\Delta_xu_j+q_ju_j=0\ee
taking the form
\bel{os} u_j=a_{j,1}e^{i\lambda \psi_1(t,x)}+a_{j,2}e^{i\lambda \psi_2(t,x)}+R_{j,\lambda}\ee
with a large parameter $\lambda>0$ and a remainder term $R_{j,\lambda}$ that admits some decay with respect to $\lambda$. The use of such a solutions, also called oscillating geometric optics solutions, goes back to \cite{RS1} who have proved unique recovery of time-independent coefficients. Since then, such approach has been used by various authors in different context including recovery of a bounded time-dependent coefficient by \cite{KiOk}. In this section we will prove how one can extend this approach, that has been specifically designed for the recovery of time-independent coefficients or  bounded time-dependent coefficients,  to the recovery of singular time-dependent coefficients.

\subsection{Oscillating geometric optics solutions}

Fixing $\omega\in\mathbb S^{n-1}$, $\lambda>1$ and $a_{j,k}\in\mathcal C^\infty(\overline{Q})$, $j=1,2$, $k=1,2$, we consider solutions of \eqref{wave1} taking the form
\bel{go1}u_1(t,x)=a_{1,1}(t,x)e^{-i\lambda (t+x\cdot\omega)}+a_{1,2}(t,x)e^{-i\lambda ((2T-t)+x\cdot\omega)}+R_{1,\lambda}(t,x),\quad (t,x)\in Q,\ee
\bel{go2}u_2(t,x)=a_{2,1}(t,x)e^{i\lambda (t+x\cdot\omega)}+a_{2,2}(t,x)e^{i\lambda (-t+x\cdot\omega)}+R_{2,\lambda}(t,x),\quad (t,x)\in Q.\ee
Here, the expression $a_{j,k}$, $j,k=1,2$, are independent of $\lambda$ and they are respectively solutions of the transport equation
\bel{go3}\pd_ta_{j,k}+(-1)^{k}\omega\cdot\nabla_xa_{j,k}=0,\quad (t,x)\in Q,\ee
and the expression $R_{j,\lambda}$, $j=1,2$, solves respectively the IBVP
\bel{eqgo1}\left\{ \begin{array}{rcll} \partial_t^2R_{1,\lambda}-\Delta_x R_{1,\lambda}+q_1R_{1,\lambda}& =  F_{1,\lambda},&  (t,x) \in Q ,\\
R_{1,\lambda}(T,x)=0,\ \ \partial_t R_{1,\lambda}(T,x)&=0,& x\in\Omega\\
 R_{1,\lambda}(t,x)=0,& \  & (t,x) \in \Sigma,& \end{array}\right.
\ee
\bel{eqgo2}\left\{ \begin{array}{rcll} \partial_t^2R_{2,\lambda}-\Delta_x R_{2,\lambda}+q_2R_{2,\lambda}& =  F_{2,\lambda},&  (t,x) \in Q ,\\
R_{2,\lambda}(0,x)=0,\ \ \partial_t R_{2,\lambda}(0,x)&=0,& x\in\Omega\\
 R_{2,\lambda}(t,x)=0,& \  & (t,x) \in \Sigma,& \end{array}\right.
\ee
with $F_{j,\lambda}=-[(\Box+q_j)(u_j-R_{j,\lambda})]$. The main point in the construction of such solutions, also called oscillating geometric optics (GO in short) solutions,  consists of proving the decay of the expression $R_{j,\lambda}$ with respect to $\lambda\to+\infty$. Actually, we would like to prove the following,
\bel{dego}\lim_{\lambda\to+\infty}\norm{R_{j,\lambda}}_{L^\infty(0,T;L^2(\Omega))}=0.\ee
For $q\in L^\infty(Q)$, the construction of GO solutions of the form \eqref{go1}-\eqref{go2}, with $a_{j,k}$ satisfying \eqref{go3} and $R_{j,\sigma}$ satisfying \eqref{eqgo1}-\eqref{dego}, has been proved in \cite[Lemma 2.2]{KiOk}. The fact that $q$ is bounded plays an important role in the arguments of \cite[Lemma 2.2]{KiOk}. For this reason we can not apply the result of \cite{KiOk} and we need to consider the following.
\begin{lem}\label{l4} Let $q_j\in L^{p_1}(0,T;L^{p_2}(\Omega))$, $j=1,2$, $p_1>2$, $p_2>n$. Then, we can find $u_j\in \mathcal K(Q)$ solving \eqref{wave1}, of the form \eqref{go1}-\eqref{go2}, with $R_{j,\lambda}$, $j=1,2$, satisfying \eqref{dego} and the following estimate
\bel{ll4a} \sup_{\lambda>1}\max_{j=1,2}\norm{R_{j,\lambda}}_{L^\infty(0,T;H^1(\Omega))}<\infty.\ee
\end{lem}
\begin{proof} We will consider this result only for $j=2$, the proof for $j=1$ being similar by symmetry. Note first that, \eqref{go3} implies that
$$\begin{aligned}F_{2,\lambda}(t,x)&=-e^{i\lambda (t+x\cdot\omega)}(\Box+q_2)a_{2,1}(t,x)-e^{i\lambda (-t+x\cdot\omega)}(\Box+q_2)a_{2,2}(t,x)\\
\ &=H_\lambda(t,x)\\
\ &=e^{i\lambda t} H_{1,\lambda}(t,x)+e^{-i\lambda t}H_{2,\lambda}(t,x),\end{aligned}$$
with
\bel{esg1}\norm{H_\lambda}_{L^2(Q)}\leq \norm{(\Box+q_2)a_{2,1}}_{L^2(Q)}+\norm{(\Box+q_2)a_{2,2}}_{L^2(Q)}.\ee
Thus, in light Proposition \ref{p1},  we have $R_{2,\lambda}\in \mathcal K(Q)$ with
\bel{esg2}\norm{R_{2,\lambda}}_{\mathcal C^1([0,T];L^2(\Omega))}+\norm{R_{2,\lambda}}_{\mathcal C([0,T];H^1(\Omega))}\leq C(1+\norm{q_2}_{L^{p_1}(0,T;L^{p_2}(\Omega))})(\norm{a_{2,1}}_{W^{2,\infty}(Q)}+\norm{a_{2,2}}_{W^{2,\infty}(Q)}).\ee
In particular, this proves \eqref{ll4a}. The only point that we need to check is the decay with respect to $\lambda$ given by \eqref{dego}.
For this purpose, we consider $v(t,x):=\int_0^tR_{2,\lambda}(s,x)ds$ and we easily check that $v$ solves
\bel{eqqgo2}\left\{ \begin{array}{rcll} \partial_t^2v-\Delta_x v =  G_\lambda,&  (t,x) \in Q ,\\
v(0,x)=0,\ \ \partial_t v(0,x)=0,& x\in\Omega\\
 v(t,x)=0, & (t,x) \in \Sigma,& \end{array}\right.
\ee
with
$$G_\lambda(t,x)=-\int_0^tq_2(s,x)R_{2,\lambda}(s,x)ds+\int_0^tH_\lambda(s,x)ds,\quad (t,x)\in Q.$$

In view of \cite[Theorem 2.1, Chapter 5]{LM2},  since $G_\lambda\in H^1(0,T;L^2(\Omega))$ we have $v\in H^2(Q)$. We define the energy $E(t)$ at time $t$ associated with $v$ and given by
$$E(t):=\int_\Omega \left(|\partial_tv|^2(t,x)+|\nabla_xv|^2(t,x)\right)dx\geq \int_\Omega |R_{2,\lambda}(t,x)|^2dx.$$
Multiplying \eqref{eqqgo2} by $\overline{\partial_tv}$ and taking the real part, we find
$$E(t)=-2\re \left(\int_0^t\int_\Omega \left(\int_0^sq_2(\tau,x)R_{2,\lambda}(\tau,x)d\tau\right)\overline{\partial_tv(s,x)}dxds\right)+2\re \left(\int_0^t\int_\Omega \left(\int_0^sH_\lambda(\tau,x)d\tau\right)\overline{\partial_tv(s,x)}dxds\right).$$
Applying Fubini's theorem, we obtain
\bel{l4a}\begin{aligned}E(t)&=-2\re \left(\int_0^t\int_\Omega q_2(\tau,x)R_{2,\lambda}(\tau,x)\left(\overline{\int_\tau^t\partial_tv(s,x)ds}\right)dxd\tau \right)+2\re \left(\int_0^t\int_\Omega \left(\int_0^sH_\lambda(\tau,x)d\tau\right)\overline{\partial_tv(s,x)}dxds\right)\\
 \ &=-2\re \left(\int_0^t\int_\Omega q_2(\tau,x)R_{2,\lambda}(\tau,x)(\overline{v(t,x)}-\overline{v(\tau,x)})dxd\tau \right)+2\re \left(\int_0^t\int_\Omega \left(\int_0^sH_\lambda(\tau,x)d\tau\right)\overline{\partial_tv(s,x)}dxds\right).\end{aligned}\ee
On the other hand, applying the H\"older inequality, we get
 $$\begin{aligned}\abs{\int_0^t\int_\Omega q_2(\tau,x)R_{2,\lambda}(\tau,x)\overline{v(t,x)}dxd\tau}&\leq \int_0^t \norm{\pd _tv(\tau,\cdot)}_{L^2(\Omega)}\norm{q_2(\tau,\cdot)v(t,\cdot)}_{L^2(\Omega)}d\tau\\
\ &\leq \left(\int_0^t \norm{\pd _tv(\tau,\cdot)}_{L^2(\Omega)}\norm{q_2(\tau,\cdot)}_{L^{p_2}(\Omega)}d\tau\right)\norm{v(t,\cdot)}_{L^{\frac{p_2-2}{2p_2}}(\Omega)}.\end{aligned}$$
Then, combining the Sobolev embedding theorem with the Poincarr\'e inequality, we deduce that
$$\begin{aligned}\abs{\int_0^t\int_\Omega q_2(\tau,x)R_{2,\lambda}(\tau,x)\overline{v(t,x)}dxd\tau}&\leq C\left(\int_0^t \norm{\pd _tv(\tau,\cdot)}_{L^2(\Omega)}\norm{q_2(\tau,\cdot)}_{L^{p_2}(\Omega)}d\tau\right)\norm{v(t,\cdot)}_{H^1(\Omega)}\\
\ &\leq C\left(\int_0^t E(\tau)^{1/2}\norm{q_2(\tau,\cdot)}_{L^{p_2}(\Omega)}d\tau\right)E(t)^{1/2}\\
\ &\leq C^2\left(\int_0^t E(\tau)^{1/2}\norm{q_2(\tau,\cdot)}_{L^{p_2}(\Omega)}d\tau\right)^2+ \frac{E(t)}{4}\\
\ &\leq C^2T\left(\int_0^t E(\tau)\norm{q_2(\tau,\cdot)}_{L^{p_2}(\Omega)}^2d\tau\right)+ \frac{E(t)}{4},\end{aligned}$$
with $C$ depending only on $\Omega$. Applying again the H\"older inequality, we get
\bel{l4b}\abs{\int_0^t\int_\Omega q_2(\tau,x)R_{2,\lambda}(\tau,x)\overline{v(t,x)}dxd\tau}\leq C\left(\int_0^t E(\tau)^{\frac{p_1}{p_1-2}}d\tau\right)^{{\frac{p_1-2}{p_1}}}\norm{q_2}_{L^{p_1}(0,T;L^{p_2}(\Omega))}^2+ \frac{E(t)}{4}.\ee
In the same way, we obtain
\bel{l4c}\begin{aligned}\abs{\int_0^t\int_\Omega q_2(\tau,x)R_{2,\lambda}(\tau,x)\overline{v(\tau,x)}dxd\tau}&\leq C \int_0^t E(\tau)\norm{q_2(\tau,\cdot)}_{L^{p_2}(\Omega)}d\tau\\
\ &\leq C\left(\int_0^t E(\tau)^{\frac{p_1}{p_1-2}}d\tau\right)^{{\frac{p_1-2}{p_1}}}\norm{q_2}_{L^{\frac{p_1}{2}}(0,T;L^{p_2}(\Omega))}.\end{aligned}\ee
Finally, fixing
$$\beta_\lambda(t,x):=\int_0^tH_\lambda(\tau,x)d\tau,$$
we find
\bel{l4d} \begin{aligned}\abs{\int_0^t\int_\Omega \left(\int_0^sH_\lambda(\tau,x)d\tau\right)\overline{\partial_tv(s,x)}dxds}&\leq \norm{\beta_\lambda}_{L^{\frac{p_1}{2}}(0,T;L^2(\Omega))}
\left(\int_0^t E(\tau)^{\frac{p_1}{2(p_1-2)}}d\tau\right)^{{\frac{p_1-2}{p_1}}}\\
\ &\leq \norm{\beta_\lambda}_{L^{\frac{p_1}{2}}(0,T;L^2(\Omega))}^2+\left(\int_0^t E(\tau)^{\frac{p_1}{2(p_1-2)}}d\tau\right)^{{\frac{2(p_1-2)}{p_1}}}\\
\ &\leq \norm{\beta_\lambda}_{L^{\frac{p_1}{2}}(0,T;L^2(\Omega))}^2+T^{{\frac{p_1-2}{p_1}}}\left(\int_0^t E(\tau)^{\frac{p_1}{p_1-2}}d\tau\right)^{{\frac{p_1-2}{p_1}}}.\end{aligned}\ee
Combining \eqref{l4a}-\eqref{l4d}, we deduce that
$$E(t)\leq \frac{E(t)}{4}+C(\norm{q_2}_{L^{p_1}(0,T;L^{p_2}(\Omega))}+1)^2\left(\int_0^t E(\tau)^{\frac{p_1}{p_1-2}}d\tau\right)^{{\frac{p_1-2}{p_1}}}+\norm{\beta_\lambda}_{L^{\frac{p_1}{2}}(0,T;L^2(\Omega))}^2$$
and we get
$$E(t)\leq C(\norm{q_2}_{L^{p_1}(0,T;L^{p_2}(\Omega))}+1)^2\left(\int_0^t E(\tau)^{\frac{p_1}{p_1-2}}d\tau\right)^{{\frac{p_1-2}{p_1}}}+\frac{4\norm{\beta_\lambda}_{L^{\frac{p_1}{2}}(0,T;L^2(\Omega))}^2}{3},$$
with $C$ depending only on $\Omega$ and $T$. Now taking the power ${\frac{p_1}{p_1-2}}$ on both side of this inequality, we get
$$E(t)^{\frac{p_1}{p_1-2}}\leq 2^{\frac{p_1}{p_1-2}}C^{\frac{p_1}{p_1-2}}(\norm{q_2}_{L^{p_1}(0,T;L^{p_2}(\Omega))}+1)^{\frac{2p_1}{p_1-2}}\int_0^t E(\tau)^{\frac{p_1}{p_1-2}}d\tau+2^{\frac{p_1}{p_1-2}}\left(\frac{4\norm{\beta_\lambda}_{L^{\frac{p_1}{2}}(0,T;L^2(\Omega))}^2}{3}\right)^{\frac{p_1}{p_1-2}}$$
and applying the Gronwall inequality, we obtain
$$E(t)^{\frac{p_1}{p_1-2}}\leq C_1\norm{\beta_\lambda}_{L^{\frac{p_1}{2}}(0,T;L^2(\Omega))}^{\frac{2p_1}{p_1-2}}e^{C_2t}\leq C_1\norm{\beta_\lambda}_{L^{\frac{p_1}{2}}(0,T;L^2(\Omega))}^{\frac{2p_1}{p_1-2}}e^{C_2T},$$
where $C_1$ depends only on $p_1$ and $C_2$ on $\norm{q_2}_{L^{p_1}(0,T;L^{p_2}(\Omega))}$, $\Omega$ and $T$. According to this estimate, the proof of the lemma will be completed if we prove that
\bel{l4e}\lim_{\lambda\to+\infty}\norm{\beta_\lambda}_{L^\infty(0,T;L^2(\Omega))}=0.\ee
This follows from some arguments similar to the end of the proof of \cite[Lemma 2.2]{KiOk} that we recall for sake of completeness. Applying the Riemann-Lebesgue lemma, for all $t\in[0,T]$ and almost every $x\in \Omega$, we have
$$\lim_{\lambda\to+\infty}\int_0^te^{i\lambda \tau}H_{1,\lambda}(\tau,x)d\tau=\lim_{\lambda\to+\infty}\int_0^te^{-i\lambda \tau}H_{2,\lambda}(\tau,x)d\tau=0.$$
Therefore, for all $t\in[0,T]$ and almost every $x\in \Omega$, we obtain
$$\lim_{\lambda\to+\infty}\beta_\lambda(t,x)=\lim_{\lambda\to+\infty}\int_0^tH_\lambda(\tau,x)d\tau=0.$$
Moreover, from the definition of $H_\lambda$, we get
$$\abs{\int_0^tH_{\lambda}(\tau,x)d\tau}\leq \int_0^t(\abs{(\Box+q_2)a_{2,1}}+\abs{(\Box+q_2)a_{2,2}}) ds,\quad t\in[0,T],\ \ x\in \Omega.$$
Thus, we deduce from Lebesgue's dominated convergence theorem
that
$$\lim_{\lambda\to+\infty}\norm{\int_0^tH_\lambda(\tau,\cdot)d\tau}_{L^2(\Omega)}=0, \quad t \in [0,T]. $$
Combining this with the estimate
\begin{align*}
&\norm{\int_0^{t_2}H_\lambda(\tau,\cdot)d\tau-\int_0^{t_1}H_\lambda(\tau,\cdot)d\tau}_{L^2(\Omega)}
\\&\qquad
\leq (t_2-t_1)^{\frac{1}{2}}[\norm{(\Box+q_2)a_{2,1}}_{L^2(Q)}+\norm{(\Box+q_2)a_{2,2}}_{L^2(Q))}],
\quad 0 \leq t_1<t_2 \leq T,
\end{align*}
we deduce \eqref{l4e}. This completes the proof the lemma. \end{proof}
\subsection{Proof of Theorem \ref{thm3} with restriction at $t=0$ or $t=T$}

In this section we will prove that \eqref{thm3a} or \eqref{thm3b} implies that $q_1=q_2$. We start by assuming that \eqref{thm3a} is fulfilled and we fix $q=q_2-q_1$ on $Q$ extended by $0$ on $\R^{1+n}\setminus Q$. We fix $\lambda>1$, $\omega\in\mathbb S^{n-1}$ and we fix $\xi\in\R^{1+n}$ satisfying $(1,-\omega)\cdot\xi=0$. Then, in view of Lemma \ref{l4}, we can consider $u_j\in\mathcal K(Q)$, $j=1,2$, solving \eqref{wave1}, of the form \eqref{go1}-\eqref{go2}, with $a_{1,1}(t,x)=(2\pi)^{\frac{n+1}{2}}e^{-i(t,x)\cdot\xi}$, $a_{1,2}=0$, $a_{2,1}=1$, $a_{2,2}=-1$ and with condition \eqref{dego}-\eqref{ll4a} fulfilled, that is,
\bel{e3}
\begin{aligned}
&u_1(t,x)=(2\pi)^{\frac{n+1}{2}} e^{-i(t,x)\cdot \xi}e^{-i\lambda(t+x\cdot w)}+R_{1,\lambda}(t,x),\\
& u_2(t,x)=e^{i\lambda(t+x\cdot w)}-e^{i\lambda(-t+x\cdot w)}+R_{2,\lambda}(t,x).
\end{aligned}\ee
Obviously, we have $u_2(0,x)=0$, since $R_{2,\lambda}(0,x)=0$ by \eqref{eqgo2}. In view of Proposition \ref{p1}, there exists a unique weak solution $v\in\mathcal{K}(Q)$ to the IBVP:
\begin{align*}
&\partial_t^2v-\Delta v+q_1v=0\quad\mbox{in}\quad Q,\\
&v|_{t=0}=u_2|_{t=0}=0,\;\partial_tv|_{t=0}=\partial_t u_2|_{t=0},\quad v|_\Sigma=u_2|_\Sigma.
\end{align*} Setting $u:=v-u_2$, we see
\bel{e1}\begin{aligned}
&\partial_t^2u-\Delta u+q_1u=(q_2-q_1)u_2\quad\mbox{in}\quad Q,\\
&u|_{t=0}=0,\;\partial_tu|_{t=0}=0,\quad u|_\Sigma=0.
\end{aligned}
\ee Noting that the inhomogeneous term $(q_2-q_1)u_2\in L^2(Q)$, due to the fact that $q_2-q_1\in L^2(0,T; L^{p_2}(\Omega))$ and $u_2\in L^\infty (0,T;H^1(\Omega))$. Hence, again using Proposition \ref{p1} gives that
$u\in \mathcal C([0,T];H^1_0(\Omega))\cap \mathcal C^1([0,T];L^2(\Omega))$.
Therefore, we have $u\in  \mathcal K(Q)\cap H_{\Box}(Q)$.
 Combining this with $u_1\in \mathcal K(Q)\cap H_\Box(Q)$, we deduce that $$(\partial_tu,-\nabla_xu),(\partial_tu_1,-\nabla_xu_1)\in H_{\textrm{div}}(Q)=\{F\in L^2(Q;\mathbb C^{n+1}):\ \textrm{div}_{(t,x)}F\in L^2(Q)\}.$$
Now, in view of \cite[Lemma 2.2]{Ka} we can multiply $u_1$ to the equation in \eqref{e1} and apply Green formula to get
\bel{e2}\begin{aligned}
&\int_Q (q_2-q_1)u_2 u_1\, dxdt\\
=&\int_Q (\partial_t^2u-\Delta u+q_1u) u_1\,dxdt\\
=&\int_Q (\Box u+q_1u) u_1-(\Box u_1+q_1u_1) u \,dxdt\\
=&\int_Q \Box u u_1-\Box u_1 u \,dxdt\\
=&\left\langle(\partial_tu,-\nabla_xu)\cdot \textbf{n}, u_1\right\rangle_{H^{-{1\over2}}(\partial Q),H^{{1\over2}}(\partial Q)}
-\left\langle(\partial_tu_1,-\nabla_xu_1)\cdot \textbf{n}, u\right\rangle_{H^{-{1\over2}}(\partial Q),H^{{1\over2}}(\partial Q)}
\end{aligned}
\ee
with $\textbf{n}$ the outward unite normal vector to $\partial Q$. Since $C_{q_1}(0)=C_{q_2}(0)$ and $v|_{t=0}=u_2|_{t=0}=0$, we see $\partial_\nu u|_\Sigma =u|_{t=T}=\partial_tu|_{t=T}=0$, in addition to the boundary conditions of $u$ in \eqref{e1}.
Consequently, it follows from \eqref{e2} that
$$
\int_Q (q_2-q_1)u_2 u_1\, dxdt=0.
$$
Inserting the expressions of $u_j$ ($j=1,2$) given by \eqref{e3}  to the previous identity gives the relation
\begin{align*}
&0=(2\pi)^{(n+1)/2}\int_Q q(t,x)e^{-i (t,x)\cdot \xi}\,dxdt+\tilde{R}_\lambda,\\
& \tilde{R}_\lambda:=(2\pi)^{(n+1)/2}\int_Q q(t,x)e^{-i (t,x)\cdot \xi} \left(-e^{-2i\lambda t}+e^{-i\lambda (t+x\cdot w)}\,R_{2,\lambda}(t,x) \right)\,dxdt\\
&\qquad\quad+\int_Q q(t,x) R_{1,\lambda}(t,x)\left(e^{i\lambda (t+x\cdot w)}-
e^{i\lambda (-t+x\cdot w)}+R_{2,\lambda}(t,x) \right)\,dxdt
\end{align*}
for all $\lambda>1$.
Using the fact that $q\in L^2(Q)$ and applying the Riemann-Lebesgue lemma and
\eqref{dego}, we deduce that
\begin{align*}
&\left|\int_Q q(t,x)e^{-i (t,x)\cdot \xi} \left(-e^{-2i\lambda t}+e^{-i\lambda (t+x\cdot w)}\,R_{2,\lambda}(t,x) \right)\,dxdt\right|\rightarrow0,\\
&\left|\int_Q q(t,x) R_{1,\lambda}(t,x)\left(e^{i\lambda (t+x\cdot w)}-
e^{i\lambda (-t+x\cdot w)}\right)\,dxdt\right|\rightarrow 0
\end{align*}
as $\lambda\rightarrow\infty$.  On the other hand, by Cauchy-Schwarz inequality
it holds that
\begin{align*}
\left|\int_Q q(t,x) R_{1,\lambda}(t,x) R_{2,\lambda}(t,x)\,dxdt\right|
&\leq  ||q\,R_{1,\lambda}||_{L^2(Q)}\,||R_{2,\lambda}||_{L^2(Q)}\\
&\leq C\, ||q||_{L^{p_1}(0,T; L^{p_2}(\Omega))}\,||R_{1,\lambda} ||_{L^{\infty}(0,T; H^{1}(\Omega))}\,||R_{2,\lambda}||_{L^{\infty}(0,T; L^{2}(\Omega))},
\end{align*}
which tends to zero
as $\lambda\rightarrow\infty$ due to the decaying behavior of $R_{j,\lambda}$ (see \eqref{dego}) and estimate \eqref{ll4a}.
 Therefore,
 $|\tilde{R}_\lambda|\rightarrow 0$ as $\lambda\rightarrow\infty$. It then follows that
\bel{tata}\mathcal Fq(\xi)=(2\pi)^{\frac{n+1}{2}}\int_{\R^{1+n}}q(t,x)e^{-i(t,x)\cdot\xi}dxdt=0.\ee
Since $\omega\in \mathbb S^{n-1}$ is arbitrary chosen, we deduce that for any $\omega\in \mathbb S^{n-1}$ and any $\xi$ lying in the  hyperplane $\{\zeta\in\R^{1+n}:\ \zeta\cdot (1,-\omega)=0\}$ of $\R^{1+n}$, the Fourier transform $\mathcal Fq$ is null at $\xi$. On the other hand, since $q\in L^{1}( \R^{1+n})$ is compactly supported in $\overline{Q}$, we know that $\mathcal Fq$ is a complex valued real-analytic function and it follows that $\mathcal Fq=0$. By inverse Fourier transform this yields the vanishing of $q$, which implies that $q_1=q_2$ in $Q$.

To prove that the relation \eqref{thm3b} implies $q_1=q_2$, we shall  consider $u_j\in\mathcal K(Q)$, $j=1,2$, solving \eqref{wave1}, of the form \eqref{go1}-\eqref{go2}, with $a_{1,1}=1$, $a_{1,2}=-1$, $a_{2,1}=(2\pi)^{\frac{n+1}{2}}e^{-i(t,x)\cdot\xi}$, $a_{2,2}=0$ and with condition \eqref{dego}-\eqref{ll4a} fulfilled. Then, by using the fact that $u_1(T,x)=0$, $x\in\Omega$, and by repeating the above arguments, we deduce that $q_1=q_2$. For brevity we omit the details.

We have proved so far that either of the conditions \eqref{thm3a} and \eqref{thm3b} implies $q_1=q_2$. It remains to prove that for $T>\textrm{Diam}(\Omega)$, the condition \eqref{thm3c} implies $q_1=q_2$.

\subsection{Proof of Theorem \ref{thm3} with restriction at $t=0$ and $t=T$}
In this section, we assume that $T>\textrm{Diam}(\Omega)$ is fulfilled and we will show that \eqref{thm3c} implies $q_1=q_2$. For this purpose, we fix $\lambda>1$, $\omega\in\mathbb S^{n-1}$ and $\epsilon=\frac{T-\textrm{Diam}(\Omega)}{4}$. We set also $\chi\in\mathcal C^\infty_0(-\epsilon,T+\textrm{Diam}(\Omega)+\epsilon)$ satisfying $\chi=1$ on $[0,T+\textrm{Diam}(\Omega)]$ and $x_0\in\overline{\Omega}$ such that
$$x_0\cdot\omega=\inf_{x\in\overline{\Omega}}x\cdot\omega.$$
We introduce the solutions $u_j\in\mathcal K(Q)$, $j=1,2$, of \eqref{wave1}, of the form \eqref{go1}-\eqref{go2}, with $$a_{1,1}(t,x)=(2\pi)^{\frac{n+1}{2}}\chi(t+(x-x_0)\cdot\omega)e^{-i(t,x)\cdot\xi},\quad a_{1,2}(t,x)=-(2\pi)^{\frac{n+1}{2}}\chi((2T-t)+(x-x_0)\cdot\omega)e^{-i(2T-t,x)\cdot\xi},$$ $$a_{2,1}(t,x)=\chi(t+(x-x_0)\cdot\omega),\quad a_{2,2}(t,x)=-\chi(-t+(x-x_0)\cdot\omega)$$ and with condition \eqref{dego}-\eqref{ll4a} fulfilled. Then, one can check that $u_1(T,x)=u_2(0,x)=0$, $x\in\Omega$, and repeating the arguments of the previous subsection we deduce that condition \eqref{thm3c} implies the orthogonality identity
\bel{thm3g}\int_Q q(t,x)u_2(t,x)u_1(t,x)dxdt=0.\ee
It remains to proves that this implies $q=0$. Note that
\bel{thm3h}\begin{aligned}\int_Q q(t,x)u_2(t,x)u_1(t,x)dxdt=&(2\pi)^{\frac{n+1}{2}}\int_{\R^{1+n}}q(t,x)\chi^2(t+(x-x_0)\cdot\omega)e^{-i(t,x)\cdot\xi}dxdt+\int_Qe^{-2i\lambda t}a_{1,1}a_{2,2}dxdt\\
\ &+\int_Qe^{-2i\lambda (T-t)}a_{1,2}a_{2,1}dxdt+e^{-2i\lambda T}\int_Qa_{1,2}a_{2,2}dxdt+\int_QZ_{\lambda}(t,x)dxdt,\end{aligned}\ee
with
$$Z_\lambda=q(u_1-R_{1,\lambda})R_{2,\lambda}+q(u_2-R_{2,\lambda})R_{1,\lambda}+qR_{2,\lambda}R_{1,\lambda}.$$
In a similar way to the previous subsection, one can check that \eqref{dego}-\eqref{ll4a} imply that
$$\lim_{\lambda\to+\infty}\int_QZ_{\lambda}dxdt=0.$$
Moreover, the Riemann-Lebesgue lemma implies
$$\lim_{\lambda\to+\infty}\left(\int_Qe^{-2i\lambda t}a_{1,1}a_{2,2}dxdt+\int_Qe^{-2i\lambda (T-t)}a_{1,2}a_{2,1}dxdt\right)=0.$$
In addition, using the fact that for $(t,x)\in Q$ we have
$$0\leq t+(x-x_0)\cdot\omega\leq T+|x-x_0|\leq T+\textrm{Diam}(\Omega),$$
we deduce that
$$q(t,x)\chi^2(t+(x-x_0)\cdot\omega)=q(t,x),\quad (t,x)\in\R^{1+n}$$
and that
$$(2\pi)^{\frac{n+1}{2}}\int_{\R^{1+n}}q(t,x)\chi^2(t+(x-x_0)\cdot\omega)e^{-i(t,x)\cdot\xi}dxdt=\mathcal Fq(\xi).$$
Thus, repeating the arguments of the previous subsection we can deduce that $q_1=q_2$ provided that
\bel{thm3i}\int_Qa_{1,2}a_{2,2}dxdt=0.\ee
Since $a_{2,2}(t,x)=-\chi(-t+(x-x_0)\cdot\omega)$ and $a_{1,2}=-(2\pi)^{\frac{n+1}{2}}\chi((2T-t)+(x-x_0)\cdot\omega)e^{-i(t,x)\cdot\xi}$, we deduce that $$\textrm{supp}(a_{2,2})\subset\{(t,x)\in \R^{1+n}:\ (x-x_0)\cdot\omega\geq t-\epsilon\},$$
$$\textrm{supp}(a_{1,2})\subset\{(t,x)\in \R^{1+n}:\ 2T-t+(x-x_0)\cdot\omega\leq T+\textrm{Diam}(\Omega)+\epsilon\}.$$
But, for any $(t,x)\in\{(t,x)\in \R^{1+n}:\ (x-x_0)\cdot\omega\geq t-\epsilon\}$, one can check that
$$2T-t+(x-x_0)\cdot\omega\geq 2T-\epsilon=T+\textrm{Diam}(\Omega)+3\epsilon>T+\textrm{Diam}(\Omega)+\epsilon.$$
Therefore, we have
$$\{(t,x)\in \R^{1+n}:\ (x-x_0)\cdot\omega\geq t-\epsilon\}\cap \{(t,x)\in \R^{1+n}:\ 2T-t+(x-x_0)\cdot\omega\leq T+\textrm{Diam}(\Omega)+\epsilon\}=\emptyset$$
and by the same way that $\textrm{supp}(a_{2,2})\cap \textrm{supp}(a_{1,2})=\emptyset$. This implies \eqref{thm3i} and by the same way that $q_1=q_2$. Thus, the proof of Theorem \ref{thm3} is completed.

\section{Proof of   Theorem \ref{thm1}}
In the previous section we have seen that the oscillating  geometric optics solutions \eqref{os} can be used for the recovery of some general singular time-dependent potential. We have even proved that, by adding a second term, we can restrict the data on the bottom $t=0$ and top $t=T$ of $Q$ while  avoiding a "reflection". Nevertheless, as mentioned in the introduction,  it is not clear how one can adapt this approach  to restrict data on the lateral boundary $\Sigma$ without requiring additional smoothness  or geometrical assumptions. In this section,  we use a different strategy for restricting the data at $\Sigma$. Namely,  we replace the oscillating GO solutions  \eqref{os} by exponentially growing and decaying solutions, of the form \eqref{exp}, in order to restrict the data on $\Sigma$ by mean of a Carleman estimate. In this section, we assume that $q_1,q_2\in L^p(Q)$, with $p>n+1$, and we will prove that \eqref{thm1a} implies $q_1=q_2$.   For this purpose, we will start with the construction of  solutions of \eqref{wave} taking the form \eqref{exp}. Then we will show  Carleman estimates for unbounded potentials and we will complete the proof of Theorem \ref{thm1}.
\subsection{Geometric optics solutions for Theorem \ref{thm1}}

Let $\omega\in\mathbb S^{n-1}$ and let $\xi\in \R^{1+n}$ be such that $\xi\cdot (1,-\omega)=0$.
In this section we consider exponentially decaying solutions $u_1\in H^1(Q)$ of the equation $(\pd_t^2-\Delta_x +q_1)u_1=0$ in $Q$ taking the form
\bel{GO1} u_1(t,x)=e^{-\lambda(t+x\cdot\omega)}(e^{-i(t,x)\cdot\xi}+w_1(t,x)),\ee
and exponentially growing solution $u_2\in H^1(Q)$ of the equation $(\pd_t^2-\Delta_x+q_2)u_2=0$ in $Q$ taking the form
\bel{GO2} u_2(t,x)=e^{\lambda (t+x\cdot\omega)}(1+w_2(t,x))\ee
where  $\lambda>1$ and the term $w_j\in H^1(Q)$, $j=1,2$, satisfies
\bel{CGO11}\norm{w_j}_{H^1(Q)}+\lambda\norm{w_j}_{L^2(Q)}\leq C,\ee
with $C$ independent of  $\lambda$.
We summarize these results in the following way.
\begin{prop}\label{p4} There exists $\lambda_1>1$ such that for $\lambda>\lambda_1$ we can find a solution $u_1\in H^1(Q)$ of $\Box u_1+q_1u_1=0$ in $Q$ taking the form \eqref{GO1}
with $w_1\in H^1(Q)$ satisfying \eqref{CGO11} for $j=1$.\end{prop}
\begin{prop}\label{p5} There exists $\lambda_2>\lambda_1$ such that for $\lambda>\lambda_2$ we can find a solution  $u_2\in H^1(Q)$ of $\Box u_2+q_2u_2=0$ in $Q$ taking the form \eqref{GO2}
with $w_2\in H^1(Q)$ satisfying \eqref{CGO11} for $j=2$.\end{prop}
We start by considering Proposition \ref{p4}. To build  solutions $u_1\in H^1(Q)$ of the form \eqref{GO1},  we first recall some preliminary tools and a suitable Carleman estimate
in Sobolev space of negative order borrowed from \cite{Ki4}. For all $m\in\R$, we introduce the space $H^m_\lambda(\R^{1+n})$ defined by
\[H^m_\lambda(\R^{1+n})=\{u\in\mathcal S'(\R^{1+n}):\ (|(\tau,\xi)|^2+\lambda^2)^{m\over 2}\hat{u}\in L^2(\R^{1+n})\},\]
with the norm
\[\norm{u}_{H^m_\lambda(\R^{1+n})}^2=\int_\R\int_{\R^n}(|(\tau,\xi)|^2+\lambda^2)^{m}|\hat{u}(\tau,\xi)|^2 d\xi d\tau.\]
Note that here we consider these spaces with $\lambda>1$ and, for $\lambda=1$, one can check that $H^m_\lambda(\R^{1+n})=H^m(\R^{1+n})$.
Here for all tempered distribution $u\in \mathcal S'(\R^{1+n})$, we denote by $\hat{u}$ the Fourier transform of $u$.
We fix the weighted operator
$$P_{\omega,\pm\lambda}:=e^{\mp \lambda(t+x\cdot\omega)}\Box e^{\pm \lambda(t+x\cdot\omega)}=\Box \pm2\lambda(\pd_t-\omega\cdot\nabla_x) $$
and we recall the following Carleman estimate
\begin{lem}\label{l1} \emph{(Lemma 5.1, \cite{Ki4})} There exists $\lambda_1'>1$ such that
\bel{l1a}\norm{v}_{L^2(\R^{1+n})}\leq C\norm{P_{\omega,\lambda}v}_{H^{-1}_\lambda(\R^{1+n})},\quad v\in\mathcal C^\infty_0(Q),\ \ \lambda>\lambda_1',\ee
with $C>0$ independent of $v$ and $\lambda$.
\end{lem}

From this result we can deduce the Carleman estimate
\begin{lem}\label{l2} Let $p_1\in (n+1,+\infty)$, $p_2\in(n,+\infty)$ and $q\in L^{p_1}(Q)\cup L^\infty(0,T;L^{p_2}(\Omega))$. Then, there exists $\lambda_1''>\lambda_1'$ such that
\bel{l2a}\norm{v}_{L^2(\R^{1+n})}\leq C\norm{P_{\omega,\lambda}v+qv}_{H^{-1}_\lambda(\R^{1+n})},\quad v\in\mathcal C^\infty_0(Q),\ \ \lambda>\lambda_1'',\ee
with $C>0$ independent of $v$ and $\lambda$.
\end{lem}

\begin{proof}
We start by considering the case $q\in L^{p_1}(Q)$.
Note first that
\bel{l2b}\norm{P_{\omega,\lambda}v+qv}_{H^{-1}_\lambda(\R^{1+n})}\geq \norm{P_{\omega,\lambda}v}_{H^{-1}_\lambda(\R^{1+n})}-\norm{qv}_{H^{-1}_\lambda(\R^{1+n})}.\ee
On the other hand, fixing $$ r=\frac{n+1}{p_1},\quad{1\over p_3}={1\over 2}-{1\over p_1}={n+1-2r\over 2(n+1)},$$ by the Sobolev embedding theorem we deduce that
$$\begin{aligned}\norm{qv}_{H^{-1}_\lambda(\R^{1+n})}&\leq \lambda^{r-1}\norm{qv}_{H^{-r}_\lambda(\R^{n+1})}\\
\ &\leq \lambda^{r-1}\norm{qv}_{H^{-r}(\R^{n+1})}\\
\ &\leq C\lambda^{r-1}\norm{qv}_{L^{\frac{p_3}{p_3-1}}(Q)}\end{aligned}$$
Combining this with the fact that
$${p_3-1\over p_3}=1-{1\over p_3}={1\over 2}+{1\over p_1},$$
we deduce from the H\"older inequality that
$$\norm{qv}_{H^{-1}_\lambda(\R^{1+n})}\leq C\lambda^{r-1}\norm{q}_{L^{p_1}(Q)}\norm{v}_{L^2(Q)}.$$
Thus, applying \eqref{l1a} and \eqref{l2b}, we deduce \eqref{l2a} for $\lambda>1$ sufficiently large.
Now let us consider the case $q\in  L^\infty(0,T;L^{p_2}(\Omega))$. Note first that
$$\norm{qv}_{H^{-1}_\lambda(\R^{1+n})}\leq \norm{qv}_{L^2(0,T;H^{-1}_\lambda(\R^{n}))}.$$
Therefore, by repeating the above arguments, we obtain
$$\norm{qv}_{H^{-1}_\lambda(\R^{1+n})}\leq C\lambda^{\frac{n}{p_2}-1}\norm{q}_{L^\infty(0,T;L^{p_2}(\Omega))}\norm{v}_{L^2(Q)}$$
which implies  \eqref{l2a} for $\lambda>1$ sufficiently large. Combining these two results, one can find $\lambda_1''>\lambda_1'$ such that \eqref{l2a} is fulfilled.\end{proof}
Using this new carleman estimate we are now in position to complete the proof of Proposition \ref{p4}.

\textbf{Proof of Proposition \ref{p4}.} Note first that
$$\begin{aligned} \Box (e^{-\lambda (t+x\cdot\omega)}e^{-i\xi\cdot (t,x)})&=[2i\lambda (1,-\omega)\cdot\xi e^{-i\xi\cdot (t,x)}+\Box e^{-i\xi\cdot (t,x)}]e^{-\lambda (t+x\cdot\omega)}\\
\ &=[\Box e^{-i\xi\cdot (t,x)}]e^{-\lambda (t+x\cdot\omega)},\end{aligned}$$
$$\Box (e^{-\lambda (t+x\cdot\omega)}w_1)=e^{-\lambda (t+x\cdot\omega)}P_{\omega,-\lambda}w_1.$$
Therefore, we need to consider $w_1\in H^1(Q)$ a solution of
\bel{p4b} P_{\omega,-\lambda}w_1+q_1w_1=-e^{\lambda (t+x\cdot\omega)}(\Box+q_1) (e^{-\lambda (t+x\cdot\omega)}e^{-i\xi\cdot (t,x)})=-(\Box+q_1) e^{-i\xi\cdot (t,x)}=F\ee
and satisfying \eqref{CGO11} for $j=1$. For this purpose, we will use estimate \eqref{l2a}. From now on, we fix $\lambda_1=\lambda_1''$.
Applying the Carleman estimate \eqref{l2a}, we define the linear form $\mathcal L$ on $\{P_{\omega,\lambda}z+q_1z:\ z\in\mathcal C^\infty_0(Q)\}$, considered as a subspace of $H^{-1}_\lambda(\R^{1+n})$ by
\[\mathcal L(P_{\omega,\lambda}z+q_1z)=\int_QzFdxdt,\quad z\in\mathcal C^\infty_0(Q).\]
Then, \eqref{l2a} implies
\[|\mathcal L(P_{\omega,\lambda}z+q_1z)|\leq C\norm{F}_{L^2(Q)}\norm{P_{\omega,\lambda}z+q_1z}_{H^{-1}_\lambda(\R^{1+n})},\quad z\in\mathcal C^\infty_0(Q).\]
Thus, by the Hahn Banach theorem we can extend $\mathcal L$ to a continuous linear form on $H^{-1}_\lambda(\R^{1+n})$ still denoted by $\mathcal L$ and satisfying $\norm{\mathcal L}\leq C\norm{F}_{L^2(Q)}$. Therefore, there exists $w_1\in H^{1}_\lambda(\R^{1+n})$ such that
\[\left\langle h,w_1\right\rangle_{H^{-1}_\lambda(\R^{1+n}),H^{1}_\lambda(\R^{1+n})}=\mathcal L(h),\quad h\in H^{-1}_\lambda(\R^{1+n}).\]
Choosing $h=P_{\omega,\lambda}z+q_1z$ with $z\in \mathcal C^\infty_0(Q)$ proves that $w_1$ satisfies $P_{\omega,-\lambda}w_1+q_1w_1=F$ in $Q$.
Moreover, we have $\norm{w_1}_{H^{1}_\lambda(\R^{1+n})}\leq \norm{\mathcal L}\leq C\norm{F}_{L^2(Q)}$. This proves that $w_1$ fufills \eqref{CGO11} which completes the proof of the proposition.\qed

Now let us consider the construction of the exponentially growing solutions given by Proposition \ref{p5}.  Combining \cite[Lemma 5.4]{Ki4} with the arguments used in Lemma \ref{l2} we obtain the Carleman estimate.

\begin{lem}\label{l3} There exists $\lambda_2'>0$ such that for $\lambda>\lambda_2'$,  we have
\bel{l3a}\norm{v}_{L^2(\R^{1+n})}\leq C\norm{P_{\omega,-\lambda}v+q_2v}_{H^{-1}_\lambda(\R^{1+n})},\quad v\in\mathcal C^\infty_0(Q),\ \ \lambda>\lambda_3',\ee
with $C>0$ independent of $v$ and $\lambda$.\end{lem}

In a similar way to Proposition \ref{p4}, we can complete the proof of Proposition \ref{p5} by applying estimate \eqref{l3a}.

\subsection{Carleman estimates for unbounded potential}

This subsection is devoted to the proof of  a Carleman estimate similar  to  \cite[Theorem 3.1]{Ki4}. More precisely,  we consider the following estimate.

\begin{Thm}\label{c1}  Let $p_1\in(n+1,+\infty)$, $p_2\in (n,+\infty)$ and assume that $q\in L^{p_1}(Q)$ \emph{(}resp $q\in L^\infty(0,T;L^{p_2}(\Omega))$\emph{)} and  $u\in\mathcal C^2(\overline{Q})$.  If $u$ satisfies the condition
 \begin{equation}\label{c1a}u_{\vert \Sigma}=0,\quad u_{\vert t=0}=\partial_tu_{\vert t=0}=0,\end{equation}
then there exists $\lambda_3>\lambda_2$ depending only on  $\Omega$, $T$ and $M\geq \norm{q}_{L^{p_1}(Q)}$ \emph{(}resp  $M\geq \norm{q}_{L^\infty(0,T;L^{p_2}(\Omega))}$\emph{)} such that the estimate
\begin{equation}\label{c1b}\begin{array}{l}\lambda \int_\Omega e^{-2\lambda( T+\omega\cdot x)}\abs{\partial_tu(T,x)}^2dx+\lambda\int_{\Sigma_{+,\omega}}e^{-2\lambda(t+\omega\cdot x)}\abs{\partial_\nu u}^2\abs{\omega\cdot\nu(x) } d\sigma(x)dt\\
+\int_Qe^{-2\lambda(t+\omega\cdot x)}[\lambda^2\abs{u}^2+|\pd_tu|^2+|\nabla_xu|^2]dxdt\\
\leq C\left(\int_Qe^{-2\lambda(t+\omega\cdot x)}\abs{(\partial_t^2-\Delta_x+q)u}^2dxdt+ \lambda^3\int_\Omega e^{-2\lambda(T+\omega\cdot x)}\abs{u(T,x)}^2dx\right)\\
\ \ \ +C\left(\lambda\int_\Omega e^{-2\lambda(T+\omega\cdot x)}\abs{\nabla_xu(T,x)}^2dx+\lambda\int_{\Sigma_{-,\omega}}e^{-2\lambda(t+\omega\cdot x)}\abs{\partial_\nu u}^2\abs{\omega\cdot\nu(x) }d\sigma(x)dt\right)\end{array}\end{equation}
holds true for $\lambda\geq \lambda_3$  with $C$  depending only on  $\Omega$, $T$ and $M$.\end{Thm}
\begin{proof} Since the proof of this result is similar for $q\in L^{p_1}(Q)$ or $q\in L^\infty(0,T;L^{p_2}(\Omega))$, we assume without lost of generality that $q\in L^{p_1}(Q)$.
Note first that for $q=0$, \eqref{c1a} follows from \cite[Theorem 3.1]{Ki4}. On the other hand, we have
$$\norm{e^{-\lambda(t+\omega\cdot x)}\abs{(\partial_t^2-\Delta_x+q)u}}_{L^2(Q)}\geq \norm{e^{-\lambda(t+\omega\cdot x)}\abs{(\partial_t^2-\Delta_x)u}}_{L^2(Q)}-\norm{e^{-\lambda(t+\omega\cdot x)}qu}_{L^2(Q)}$$
and by the H\"older inequality we deduce that
$$\norm{e^{-\lambda(t+\omega\cdot x)}\abs{(\partial_t^2-\Delta_x+q)u}}_{L^2(Q)}\geq \norm{e^{-2\lambda(t+\omega\cdot x)}\abs{(\partial_t^2-\Delta_x)u}}_{L^2(Q)}-\norm{q}_{L^{p_1}(Q)}\norm{e^{-\lambda(t+\omega\cdot x)}u}_{L^{p_3}(Q)}$$
with $p_3=\frac{2p_1}{p_1-2}$. Now fix $s:=\frac{n+1}{p_1}\in(0,1)$ and notice that
$${1\over p_3}={n+1-2s\over 2(n+1)}.$$
Thus, by the Sobolev embedding theorem, we have
$$\norm{e^{-\lambda(t+\omega\cdot x)}u}_{L^{p_3}(Q)}\leq C\norm{e^{-\lambda(t+\omega\cdot x)}u}_{H^s(Q)}$$
and by interpolation we deduce that
$$\begin{aligned}\norm{e^{-\lambda(t+\omega\cdot x)}u}_{L^{p_3}(Q)}&\leq C\left(\int_Q e^{-2\lambda(t+\omega\cdot x)}(\lambda^2 |u|^2+|\partial_tu|^2+|\nabla_x u|^2dxdt\right)^{\frac{s}{2}}\norm{e^{-\lambda(t+\omega\cdot x)}u}_{L^2(Q)}^{1-s}\\
&\leq C\lambda^{s-1} \left(\int_Q e^{-2\lambda(t+\omega\cdot x)}(\lambda^2 |u|^2+|\partial_tu|^2+|\nabla_x u|^2dxdt\right)^{\frac{1}{2}}.\end{aligned}$$
On the other hand, in view of \cite[Theorem 3.1]{Ki4}, there exists $\lambda_3'>1$ such that, for $\lambda>\lambda_3'$, we have
$$\begin{array}{l}\int_Qe^{-2\lambda(t+\omega\cdot x)}(\lambda^2 |u|^2+|\partial_tu|^2+|\nabla_x u|^2)dxdt\\
\leq C\left(\int_Qe^{-2\lambda(t+\omega\cdot x)}\abs{(\partial_t^2-\Delta_x)u}^2dxdt+ \lambda^3\int_\Omega e^{-2\lambda(T+\omega\cdot x)}\abs{u(T,x)}^2dx\right)\\
\ \ \ +C\left(\lambda\int_\Omega e^{-2\lambda(T+\omega\cdot x)}\abs{\nabla_xu(T,x)}^2dx+\lambda\int_{\Sigma_{-,\omega}}e^{-2\lambda(t+\omega\cdot x)}\abs{\partial_\nu u}^2\abs{\omega\cdot\nu(x) }d\sigma(x)dt\right).\end{array}$$
Thus, we get
$$\begin{array}{l}\int_Qe^{-2\lambda(t+\omega\cdot x)}\abs{(\partial_t^2-\Delta_x+q)u}^2dxdt+ \lambda^3\int_\Omega e^{-2\lambda(T+\omega\cdot x)}\abs{u(T,x)}^2dx\\
\ \ \ +\lambda\int_\Omega e^{-2\lambda(T+\omega\cdot x)}\abs{\nabla_xu(T,x)}^2dx+\lambda\int_{\Sigma_{-,\omega}}e^{-2\lambda(t+\omega\cdot x)}\abs{\partial_\nu u}^2\abs{\omega\cdot\nu(x) }d\sigma(x)dt\\
\geq \frac{1}{2}\int_Qe^{-2\lambda(t+\omega\cdot x)}\abs{(\partial_t^2-\Delta_x)u}^2dxdt+ \lambda^3\int_\Omega e^{-2\lambda(T+\omega\cdot x)}\abs{u(T,x)}^2dx\\
\ \ \ +\lambda\int_\Omega e^{-2\lambda(T+\omega\cdot x)}\abs{\nabla_xu(T,x)}^2dx+\lambda\int_{\Sigma_{-,\omega}}e^{-2\lambda(t+\omega\cdot x)}\abs{\partial_\nu u}^2\abs{\omega\cdot\nu(x) }d\sigma(x)dt\\
-C\norm{q}_{L^{p_1}(Q)}^2\lambda^{2(s-1)}\left(\int_Qe^{-2\lambda(t+\omega\cdot x)}\abs{(\partial_t^2-\Delta_x)u}^2dxdt+ \lambda^3\int_\Omega e^{-2\lambda(T+\omega\cdot x)}\abs{u(T,x)}^2dx\right)\\
-C\norm{q}_{L^{p_1}(Q)}^2\lambda^{2(s-1)}\left(\lambda\int_\Omega e^{-2\lambda(T+\omega\cdot x)}\abs{\nabla_xu(T,x)}^2dx+\lambda\int_{\Sigma_{-,\omega}}e^{-2\lambda(t+\omega\cdot x)}\abs{\partial_\nu u}^2\abs{\omega\cdot\nu(x) }d\sigma(x)dt\right).\end{array}$$
Therefore, fixing $\lambda$ sufficiently large and applying \cite[Theorem 3.1]{Ki4} with $a=q=0$ we deduce \eqref{c1b}.

\end{proof}

\begin{rem}\label{rr} Note that, by density, \eqref{c1b} remains true for  $u\in\mathcal C^1([0,T]; L^2(\Omega))\cap \mathcal C([0,T]; H^1(\Omega))$ satisfying \eqref{c1a}, $(\partial_t^2-\Delta_x)u\in L^2(Q)$ and $\partial_\nu u\in L^2(\Sigma)$.\end{rem}

\subsection{Completion of the proof of Theorem \ref{thm1}}
This subsection is devoted to the proof of Theorem \ref{thm1}.  From now on, we set $q=q_2-q_1$ on $Q$ and  we assume  that $q=0$ on $\R^{1+n}\setminus Q$. For all $\theta\in\mathbb S^{n-1}$ and all  $r>0$, we set
\[\partial\Omega_{+,r,\theta}=\{x\in\partial\Omega:\ \nu(x)\cdot \theta>r\},\quad\partial\Omega_{-,r,\theta}=\{x\in\partial\Omega:\ \nu(x)\cdot \theta\leq r\}\]
and $\Sigma_{\pm,r,\theta}=(0,T)\times \partial\Omega_{\pm,r,\theta}$. Here and in the remaining of this text we always assume, without mentioning it, that $\theta$ and $r$ are chosen in such way that $\partial\Omega_{\pm,r,\pm \theta}$ contain  a non-empty relatively open subset of $\partial\Omega$.
Without lost of generality we assume that there exists $\epsilon>0$ such that  for all $\omega\in\{\theta\in\mathbb S^{n-1}:|\theta-\omega_0|\leq\epsilon\} $ we have $\partial\Omega_{-,\epsilon,\omega}\subset V'$. In order to prove Theorem \ref{thm1}, we will use the Carleman estimate stated in Theorem \ref{c1}. Let   $\lambda >\lambda_3$ and fix $\omega\in\{\theta\in\mathbb S^{n-1}:|\theta-\omega_0|\leq\epsilon\} $. According to Proposition \ref{p4}, we can introduce
\[u_1(t,x)=e^{-\lambda(t+x\cdot\omega)}(e^{-i(t,x)\cdot\xi}+w_1(t,x)),\ (t,x) \in Q,\]
where $u_1\in H^1(Q)$ satisfies $\partial_t^2u_1-\Delta_xu_1+q_1u_1=0$,  $\xi\cdot(1,-\omega)=0$ and  $w_1$ satisfies \eqref{CGO11} for $j=1$. Moreover, in view of Proposition \ref{p5}, we consider $u_2\in H^1(Q)$ a solution of $\partial_t^2u_2-\Delta_xu_2+q_2u_2=0$, of the form
\[u_2(t,x)=e^{\lambda (t+x\cdot\omega)}(1+w_2(t,x)),\quad (t,x)\in Q,\]
where  $w_2$ satisfies \eqref{CGO11} for $j=2$.
In view of Proposition \ref{p2}, there exists a unique weak solution $z_1\in H_\Box (Q)$ of
 \bel{eq3}
\left\{
\begin{array}{ll}
 \partial_t^2z_1-\Delta_xz_1 +q_1z_1=0 &\mbox{in}\ Q,
\\

\tau_{0}z_1=\tau_{0}u_2. &

\end{array}
\right.
\ee
Then, $u=z_1-u_2$ solves
  \bel{eq4}
\left\{\begin{array}{ll}
 \partial_t^2u-\Delta_xu +q_1u=(q_2-q_1)u_2 &\mbox{in}\ Q,
\\
u(0,x)=\partial_tu(0,x)=0 & \mathrm{on}\  \Omega,\\

u=0 &\mbox{on}\ \Sigma.
\end{array}\right.
\ee
Since $u_2\in H^1(Q)$, by the Sobolev embedding theorem we have $(q_2-q_1)u_2\in L^2(Q)$. Thus, repeating the arguments of Theorem \ref{thm3}, we derive the formula \eqref{e2}. On the other hand, we have $u_{|t=0}=\partial_tu_{|t=0}=u_{|\Sigma}=0$ and condition \eqref{thm1a} implies that $u_{|t=T}=\partial_\nu u_{|V}=0$. In addition, in view of \cite[Theorem 2.1]{LLT}, we have  $\partial_\nu u\in L^2(\Sigma)$.  Combining this with the fact that $u\in \mathcal C^1([0,T];L^2(\Omega))$ and $u_1\in H^1(Q)\subset H^1(0,T;L^2(\Omega))$, we obtain
\begin{equation}\label{t3a} \int_Qqu_2u_1dxdt=-\int_{\Sigma\setminus V}\partial_\nu uu_1d\sigma(x)dt+\int_\Omega \partial_tu(T,x)u_1(T,x)dx.\end{equation}
Applying  the Cauchy-Schwarz inequality to the first expression on the right hand side of this formula and using the fact that  $(\Sigma\setminus V)\subset {\Sigma}_{+,\epsilon,\omega}$, we get
\[\begin{aligned}\abs{\int_{\Sigma\setminus V}\partial_\nu uu_1d\sigma(x)dt}&\leq\int_{{\Sigma}_{+,\epsilon,\omega}}\abs{\partial_\nu ue^{-\lambda(t+\omega\cdot x)}(e^{-i(t,x)\cdot\xi}+w_1)} d\sigma(x)dt \\
 \ &\leq C(1+\norm{w_1}_{L^2(\Sigma)})\left(\int_{{\Sigma}_{+,\epsilon,\omega}}\abs{e^{-\lambda(t+\omega\cdot x)}\partial_\nu u}^2d\sigma(x)dt\right)^{\frac{1}{2}},\end{aligned}\]
for some $C$ independent of $\lambda$. On the other hand, one can check that
\[\norm{w_1}_{L^2(\Sigma))}\leq C\norm{w_1}_{H^1(Q)}.\]
Combining this with \eqref{CGO11}, we obtain
$$\abs{\int_{\Sigma\setminus V}\partial_\nu uu_1d\sigma(x)dt}\leq C \left(\int_{{\Sigma}_{+,\epsilon,\omega}}\abs{e^{-\lambda(t+\omega\cdot x)}\partial_\nu u}^2d\sigma(x)dt\right)^{\frac{1}{2}}.$$
 In the same way, we have
\[\begin{aligned}\abs{\int_\Omega \partial_tu(T,x)u_1(T,x)dx}&\leq\int_{\Omega}\abs{\partial_t u(T,x)e^{-\lambda(T+\omega\cdot x)}\left(e^{-i(T,x)\cdot\xi}+w_1(T,x)\right)}dx \\
 \ &\leq C \left(\int_{\Omega}\abs{e^{-\lambda(T+\omega\cdot x)}\partial_t u(T,x)}^2dx\right)^{\frac{1}{2}}.\end{aligned}\]
Combining these estimates with the Carleman estimate \eqref{c1b} and applying the fact that $u_{|t=T}=\partial_\nu u_{|\Sigma_{-,\omega}}=0$, ${\partial\Omega}_{+,\epsilon,\omega}\subset {\partial\Omega}_{+,\omega}$, we find
\begin{eqnarray}\label{loi}&&\abs{\int_Qqu_2u_1dxdt}^2\cr
&&\leq C\left(\int_{{\Sigma}_{+,\epsilon,\omega}}\abs{e^{-\lambda(t+\omega\cdot x)}\partial_\nu u}^2d\sigma(x)dt+\int_\Omega \abs{e^{-\lambda(T+\omega\cdot x)}\partial_tu(T,x)}^2dx\right)\cr
&&\leq \epsilon^{-1}C\lambda^{-1}\left(\lambda\int_{{\Sigma}_{+,\omega}}\abs{e^{-\lambda(t+\omega\cdot x)}\partial_\nu u}^2\omega\cdot \nu(x)d\sigma(x)dt+\lambda\int_\Omega \abs{e^{-\lambda(T+\omega\cdot x)}\partial_tu(T,x)}^2dx\right)\cr
&&\leq \epsilon^{-1}C\lambda^{-1}\left(\int_Q\abs{ e^{-\lambda(t+\omega\cdot x)}(\partial_t^2-\Delta_x+q_1)u}^2dxdt\right)\cr
&&\leq \epsilon^{-1}C\lambda^{-1}\left(\int_Q\abs{ e^{-\lambda(t+\omega\cdot x)}qu_2}^2dxdt\right)\cr
&&\leq \epsilon^{-1}C\lambda^{-1}\left(\int_Q\abs{ q(1+w_2)}^2dxdt\right)
.\end{eqnarray}
Here $C>0$ stands for some generic constant independent of $\lambda$. On the other hand, in a similar way to Lemma \ref{l2}, combining the H\"older inequality and the Sobolev embedding theorem we get
$$\int_Q\abs{ q(1+w_2)}^2dxdt\leq C\norm{q}_{L^{p}(Q)}^2\norm{(1+w_2)}_{H^{\frac{n+1}{p}}(Q)}^2\leq C\norm{q}_{L^{p}(Q)}^2(1+\norm{w_2}_{H^1(Q)})^2.$$
Combining this with \eqref{CGO11} and \eqref{loi}, we obtain
\[\abs{\int_Qqu_2u_1dxdt}\leq C\lambda^{-1/2}.\]
It follows
\begin{equation}\label{t3cc}\lim_{\lambda\to+\infty}\int_Qqu_2u_1dxdt=0.\end{equation}
Moreover, \eqref{GO1}-\eqref{CGO11} imply
$$\int_Qqu_2u_1dxdt=\int_{\R^{1+n}}q(t,x)e^{-i\xi\cdot(t,x)}dxdt+\int_{Q}W_\lambda(t,x)dxdt,$$
with
\[\int_{Q}|W_\lambda(t,x)|dxdt\leq C\lambda^{-1}.\]
Combining this with \eqref{t3cc}, for all $\omega\in\{y\in\mathbb S^{n-1}:|y-\omega_0|\leq\epsilon\} $ and  all $\xi\in(1,-\omega)^\bot:=\{\zeta\in\R^{1+n}:\ \zeta\cdot(1,-\omega)=0\}$, the Fourier transform $\mathcal F(q)$ of $q$ satisfies $\mathcal F(q)(\xi)=0$. On the other hand, since $q\in L^1(\R^{1+n})$ is  supported on $\overline{Q}$ which is compact,  $\mathcal F(q)$ is a complex valued real-analytic function and it follows that $q=0$ and $q_1=q_2$. This completes the  proof of Theorem \ref{thm1}.

\section{Proof of   Theorem \ref{thm2}}
Let us first remark that, in contrast to Theorem \ref{thm3}, in Theorem \ref{thm1} we do not restrict the data to solutions of \eqref{wave} satisfying $u_{|t=0}=0$. In this section we will show  Theorem \ref{thm2} by combining the restriction on the bottom $t=0$, the top $t=T$ of $Q$ stated in Theorem \ref{thm3} with the restriction on the lateral boundary $\Sigma$ stated in Theorem \ref{thm1}. From now on, we fix $q_1,q_2\in L^\infty(0,T;L^p(\Omega))$, $p>n$, and we will show that condition \eqref{thm2a} implies $q_1=q_2$. For this purpose we still consider  exponentially growing and decaying GO solutions close to those of the previous subsection, but this time we need to take into account the constraint $u_2(0,x)=0$ required in
Theorem \ref{thm2}. For this purpose, we will consider a   different construction comparing to the one of the previous section which will follow from a Carleman estimate in negative order Sobolev space only with respect to the space variable.
\subsection{Carleman estimate in negative Sobolev space  for Theorem \ref{thm2}}
In this subsection we will  derive a Carleman estimate in negative order Sobolev space which will be one of the main tool for the construction of exponentially growing solutions $u_2$ of \eqref{wave1} taking the form
\bel{GOO1}u_2(t,x)=e^{\lambda(t+x\cdot\omega)}(1+w_2(t,x))-e^{\lambda(-t+x\cdot\omega)},\ee
 with the restriction $\tau_{0,2}u_2=0$ (recall that for $v\in\mathcal C^\infty(\overline{Q})$, $\tau_{0,2}v=v_{|t=0}$). In a similar way to the previous section, for all $m\in\R$, we introduce the space $H^m_\lambda(\R^{n})$ defined by
\[H^m_\lambda(\R^{n})=\{u\in\mathcal S'(\R^{n}):\ (|\xi|^2+\lambda^2)^{m\over 2}\hat{u}\in L^2(\R^{n})\},\]
with the norm
\[\norm{u}_{H^m_\lambda(\R^{n})}^2=\int_{\R^n}(|\xi|^2+\lambda^2)^{m}|\hat{u}(\xi)|^2 d\xi .\]
In order to construct solutions $u_2$ of the form \eqref{GO2} and satisfying $\tau_{0,2}u_2=0$, instead of the Carleman estimate \eqref{l3a}, we consider the following.

\begin{Thm}\label{t1} There exists $\lambda_2'>0$ such that for $\lambda>\lambda_2'$ and for all $v\in\mathcal C^2([0,T];\mathcal C^\infty_0(\Omega))$ satisfying
\bel{t1a}v(T,x)=\pd_tv(T,x)= v(0,x)=0,\quad x\in\R^n,\ee
  we have
\bel{t1b}\norm{v}_{L^2((0,T)\times\R^{n})}\leq C\norm{P_{\omega,-\lambda}v+q_2v}_{L^2(0,T;H^{-1}_\lambda(\R^{n}))},\ee
with $C>0$ independent of $v$ and $\lambda$.\end{Thm}
In order to prove this theorem, we start by recalling the following intermediate tools. From now on, for $m\in\R$ and $\xi\in \R^n$,  we set $$\left\langle \xi,\lambda\right\rangle=(|\xi|^2+\lambda^2)^{1\over2}$$
and $\left\langle D_x,\lambda\right\rangle^m u$ defined by
\[\left\langle D_x,\lambda\right\rangle^m u=\mathcal F^{-1}(\left\langle \xi,\lambda\right\rangle^m \mathcal Fu).\]
For $m\in\R$ we define also the class of symbols
\[S^m_\lambda=\{c_\lambda\in\mathcal C^\infty(\R^{n}\times\R^{n}):\ |\pd_x^\alpha\pd_\xi^\beta c_\lambda(x,\xi)|\leq C_{\alpha,\beta}\left\langle \xi,\lambda\right\rangle^{m-|\beta|},\  \alpha,\beta\in\mathbb N^n\}.\]
Following \cite[Theorem 18.1.6]{Ho3}, for any $m\in\R$ and $c_\lambda\in S^m_\lambda$, we define $c_\lambda(x,D_x)$, with  $D_x=-i\nabla_x$, by
\[c_\lambda(x,D_x)u(x)=(2\pi)^{-{n\over 2}}\int_{\R^{n}}c_\lambda(x,\xi)\hat{u}(\xi)e^{ix\cdot \xi} d\xi.\]
For all $m\in\R$, we  define $OpS^m_\lambda:=\{c_\lambda(x,D_x):\ c_\lambda\in S^m_\lambda\}$ and for $m=-\infty$ we set
$$OpS^{-\infty}_\lambda:=\bigcap_{m\in\R}OpS^m_\lambda.$$
Now let us consider the following intermediate result.
\begin{lem}\label{l5} There exists $\lambda_2''>0$ such that for $\lambda>\lambda_2''$ and for all $v\in\mathcal C^2([0,T];\mathcal C^\infty_0(\Omega))$ satisfying \eqref{t1a},  we have
\bel{l5a}\norm{v}_{L^2(0,T;H^1_\lambda(\R^{n}))}\leq C\norm{P_{\omega,-\lambda}v}_{L^2((0,T)\times\R^{n}))},\ee
with $C>0$ independent of $v$ and $\lambda$.\end{lem}
\begin{proof}
Consider $w(t,x)=v(T-t,x)$ and note that according to \eqref{t1a}, we have $w\in\mathcal C^2([0,T];\mathcal C^\infty_0(\Omega))$ and
$$w(0,x)=\pd_tw(0,x)= w(T,x)=0,\quad x\in\R^n.$$
Therefore, in a similar way to the proof of \cite[Lemma 4.1]{Ki2}, one can check that
\bel{l5b}\int_Q|P_{-\omega,\lambda}w|^2dxdt\geq \int_Q|\Box w|^2dxdt+c\lambda^2\int_Q| w|^2dxdt ,\ee
with $c>0$ independent of $w$ and $\lambda$. Now, recalling that $w$ solves
$$\left\{ \begin{array}{rcll} \partial_t^2w-\Delta_x w =  \Box w,&  (t,x) \in Q ,\\

w(0,x)=0,\ \ \partial_t w(0,x)=0,& x\in\Omega\\
 w(t,x)=0&(t,x)\in\Sigma, \end{array}\right.$$
we deduce that
$$\int_Q|\nabla_x w|^2dxdt\leq C\int_Q|\Box w|^2dxdt,$$
where $C$ depends only on $T$ and $\Omega$. Combining this with \eqref{l5b}, we obtain
$$\norm{w}_{L^2(0,T;H^1_\lambda(\R^{n}))}^2\leq C\int_Q|P_{-\omega,\lambda}w|^2dxdt.$$
Using the fact  that $P_{-\omega,\lambda}w(t,x)=P_{\omega,-\lambda}v(T-t,x)$, we deduce \eqref{l5a}.\end{proof}

Armed with this Carleman estimate, we are now in position of completing the proof of Theorem \ref{t1}.
\textbf{Proof of Theorem \ref{t1}.} Let $v\in\mathcal C^2([0,T];\mathcal C^\infty_0(\Omega))$ satisfy \eqref{t1a}, consider $\Omega_j$, $j=1,2$, two bounded open smooth domains of $\R^n$ such that $\overline{\Omega} \subset\Omega_1$, $\overline{\Omega_1} \subset\Omega_2$ and let $\psi\in\mathcal C^\infty_0(\Omega_2)$ be such that $\psi=1$ on   $\overline{\Omega_1}$. We consider $w\in\mathcal C^2([0,T];\mathcal C^\infty_0(\overline{\Omega}))$ given by
$$w(t,\cdot)=\psi \left\langle D_x,\lambda\right\rangle^{-1}v(t,\cdot)$$
and we remark that $w$ satisfies
\bel{t1c}w(T,x)=\pd_tw(T,x)= w(0,x)=0,\quad x\in\R^n.\ee
Now let us consider the quantity $\left\langle D_x,\lambda\right\rangle^{-1}P_{\omega,-\lambda}\left\langle D_x,\lambda\right\rangle w$.
Note first that
$$\norm{P_{\omega,-\lambda}\left\langle D_x,\lambda\right\rangle w}_{L^2(0,T;H^{-1}_\lambda(\R^n))}=\norm{\left\langle D_x,\lambda\right\rangle^{-1}P_{\omega,-\lambda}\left\langle D_x,\lambda\right\rangle w}_{L^2((0,T)\times\R^n)}.$$
Moreover, it is clear that
$$\left\langle D_x,\lambda\right\rangle^{-1}P_{\omega,-\lambda}\left\langle D_x,\lambda\right\rangle=P_{\omega,-\lambda}.$$
Therefore, we have
$$\norm{P_{\omega,-\lambda}\left\langle D_x,\lambda\right\rangle w}_{L^2(0,T;H^{-1}_\lambda(\R^n))}=\norm{P_{\omega,-\lambda} w}_{L^2((0,T)\times\R^n)}$$
and, since $w$ satisfies \eqref{t1c}, combining this with \eqref{l5a}  we deduce that
\bel{t1d}\norm{w}_{L^2(0,T;H^1_\lambda(\R^{n}))}\leq C\norm{P_{\omega,-\lambda}\left\langle D_x,\lambda\right\rangle w}_{L^2(0,T;H^{-1}_\lambda(\R^n))}.\ee
On the other hand, fixing $\psi_1\in \mathcal C^\infty_0(\Omega_1)$ satisfying $\psi_1=1$ on $\overline{\Omega}$, we get
$$w(t,\cdot)=\left\langle D_x,\lambda\right\rangle^{-1}v(t,\cdot)+(\psi-1)\left\langle D_x,\lambda\right\rangle^{-1}\psi_1v(t,\cdot)$$
and, combining this with \eqref{t1d}, we deduce  that
\bel{t1e}\begin{array}{l}\norm{v}_{L^2((0,T)\times \R^n)}\\
=\norm{\left\langle D_x,\lambda\right\rangle^{-1}v}_{L^2(0,T;H^{1}_\lambda(\R^n))}\\
\leq \norm{w}_{L^2(0,T;H^1_\lambda(\R^{n}))}+\norm{(\psi-1)\left\langle D_x,\lambda\right\rangle^{-1}\psi_1v}_{L^2(0,T;H^{1}_\lambda(\R^n))}\\
\leq C\norm{P_{\omega,-\lambda}\left\langle D_x,\lambda\right\rangle w}_{L^2(0,T;H^{-1}_\lambda(\R^n))}+\norm{(\psi-1)\left\langle D_x,\lambda\right\rangle^{-1}\psi_1v}_{L^2(0,T;H^{1}_\lambda(\R^n))}\\
\leq C\norm{P_{\omega,-\lambda}v}_{L^2(0,T;H^{-1}_\lambda(\R^n))}+C\norm{P_{\omega,-\lambda}\left\langle D_x,\lambda\right\rangle (\psi-1)\left\langle D_x,\lambda\right\rangle^{-1}\psi_1v}_{L^2(0,T;H^{-1}_\lambda(\R^n))}\\
\ \ \ +\norm{(\psi-1)\left\langle D_x,\lambda\right\rangle^{-1}\psi_1v}_{L^2(0,T;H^{1}_\lambda(\R^n))}.\end{array}\ee
Moreover, since $(\psi-1)=0$ on neighborhood of supp$(\psi_1)$, in view of \cite[Theorem 18.1.8]{Ho3},  we have $(\psi-1) \left\langle D_x,\lambda\right\rangle^{-1}\psi_1\in OpS^{-\infty}_\lambda$. In the same way,  \cite[Theorem 18.1.8]{Ho3} implies that
$$P_{\omega,-\lambda}\left\langle D_x,\lambda\right\rangle (\psi-1)\left\langle D_x,\lambda\right\rangle^{-1}\psi_1\in OpS^{-\infty}_\lambda$$
and we deduce that
$$\begin{array}{l}C\norm{P_{\omega,-\lambda}\left\langle D_x,\lambda\right\rangle (\psi-1)\left\langle D_x,\lambda\right\rangle^{-1}\psi_1v}_{L^2(0,T;H^{-1}_\lambda(\R^n))}+\norm{(\psi-1)\left\langle D_x,\lambda\right\rangle^{-1}\psi_1v}_{L^2(0,T;H^{-1}_\lambda(\R^n))}\\
\leq \frac{\norm{v}_{L^2((0,T)\times \R^n)}}{\lambda^2}.\end{array}$$
Combining this with \eqref{t1e} and choosing $\lambda$ sufficiently large, we deduce \eqref{t1b} for $q_2=0$. Then, we deduce \eqref{t1b} for $q_2\neq0$ by applying arguments similar to Lemma \ref{l2}.\qed

Applying the Carleman estimate \eqref{t1b}, we can now build solutions $u_2$ of the form \eqref{GOO1} and satisfying $\tau_{0,2}u_2=0$ and complete the proof of Theorem \ref{thm2}.
\subsection{Completion of the proof of   Theorem \ref{thm2}}
We start by proving existence of a solution $u_2\in L^2(0,T;H^1(\Omega))$ of the form \eqref{GOO1} with the term $w_2\in L^2(0,T;H^1(\Omega))\cap e^{-\lambda (t+x\cdot\omega)}H_{\Box,*}(Q)$, satisfying
\bel{CGO01}\norm{w_2}_{L^2(0,T;H^1(\Omega))}+\lambda\norm{w_2}_{L^2(Q)}\leq C,\ee
\bel{CGO02}\tau_{0,2}w_2=0.\ee
This result is summarized in the following way.
\begin{prop}\label{p6} There exists $\lambda_2>\lambda_1$ such that for $\lambda>\lambda_2$ we can find a solution  $u_2\in L^2(0,T;H^1(\Omega))$ of $\Box u_2+q_2u_2=0$ in $Q$ taking the form \eqref{GOO1}
with $w_2\in L^2(0,T;H^1(\Omega))\cap e^{-\lambda (t+x\cdot\omega)}H_{\Box,*}(Q)$ satisfying \eqref{CGO01}-\eqref{CGO02}.\end{prop}
\begin{proof} We need to consider $w_2 \in L^2(0,T;H^1(\Omega))$ a solution of
\bel{p6b} P_{\omega,\lambda}w_2+q_2w_2=-e^{-\lambda (t+x\cdot\omega)}(\Box+q_2) (e^{\lambda (t+x\cdot\omega)}-e^{\lambda (-t+x\cdot\omega)})=-q_2(1-e^{-2\lambda t}),\ee
satisfying \eqref{CGO01}-\eqref{CGO02}.
Note that here, we use \eqref{p6b} and the fact that $P_{\omega,\lambda}w_2=e^{-\lambda (t+x\cdot\omega)}\Box e^{\lambda (t+x\cdot\omega)}w_2$ in order to prove  that $w_2\in e^{-\lambda (t+x\cdot\omega)}H_{\Box,*}(Q)$ and we  define $\tau_{0,2}w_2$ by $\tau_{0,2}w_2=e^{-\lambda x\cdot\omega}\tau_{0,2}e^{\lambda (t+x\cdot\omega)}w_2$.
We will construct such a function $w_2$ by applying estimate \eqref{t1b}. From now on, we fix $\lambda_2=\lambda_2''$.
Applying the Carleman estimate \eqref{t1b}, we define the linear form $\mathcal M$ on $$\mathcal I=\{P_{\omega,-\lambda}v+q_2v:\ v\in\mathcal C^2([0,T];\mathcal C^\infty_0(\Omega))\ \textrm{ satisfying \eqref{t1a}}\},$$ considered as a subspace of $L^2(0,T;H^{-1}_\lambda(\R^{n}))$, by
\[\mathcal M(P_{\omega,-\lambda}v+q_2v)=-\int_Qvq_2(1-e^{-2\lambda t})dxdt,\quad v\in\mathcal I.\]
Then, \eqref{t1b} implies
\[|\mathcal M(P_{\omega,-\lambda}v+q_1v)|\leq C\norm{q_2}_{L^2(Q)}\norm{P_{\omega,-\lambda}v+q_1v}_{L^2(0,T;H^{-1}_\lambda(\R^{n})},\quad v\in\mathcal I,\]
with $C>0$ independent of $\lambda$ and $v$.
Thus, by the Hahn Banach theorem we can extend $\mathcal M$ to a continuous linear form on $L^2(0,T;H^{-1}_\lambda(\R^{n}))$ still denoted by $\mathcal M$ and satisfying $\norm{\mathcal M}\leq C\norm{q_2}_{L^2(Q)}$. Therefore, there exists $w_2\in L^2(0,T;H^{1}_\lambda(\R^{n}))$ such that
\[\left\langle g,w_2\right\rangle_{L^2(0,T;H^{-1}_\lambda(\R^{n}),L^2(0,T;H^{1}_\lambda(\R^{n}))}=\mathcal M(g),\quad g\in L^2(0,T;H^{-1}_\lambda(\R^{n}).\]
Choosing $g=P_{\omega,-\lambda}v+q_2v$ with $v\in \mathcal C^\infty_0(Q)$ proves that $w_2$ satisfies $P_{\omega,\lambda}w_2+q_2w_1=-q_2(1-e^{-2\lambda t})$ in $Q$. Moreover, choosing $g=P_{\omega,-\lambda}v+q_1v$, with $v\in\mathcal I$ and $\pd_tv_{|t=0}$  arbitrary, proves that \eqref{CGO02} is fulfilled. Finally, using the fact that $\norm{w_2}_{L^2(0,T;H^{1}_\lambda(\R^{n}))}\leq \norm{\mathcal M}\leq C\norm{q}_{L^2(Q)}$ proves that $w_2$ fulfills \eqref{CGO01} which completes the proof of the proposition.\end{proof}
Using this proposition, we are now in position to complete the proof of Theorem \ref{thm2}.

\textbf{Proof of Theorem \ref{thm2}.} Let us remark that since Lemma \ref{l2} and Theorem \ref{c1} are valid when $q\in L^\infty(0,T;L^{p}(\Omega))$ one can easily extend Proposition \ref{p4} to the case $q_1\in L^\infty(0,T;L^{p}(\Omega))$. Therefore, in the context of this section, Proposition \ref{p4} holds true.
Combining Proposition \ref{p4} with Proposition \ref{p6}, we deduce existence of a solution $u_1\in H^1(Q)$ of $\Box u_1+q_1u_1=0$ in $Q$ taking the form \eqref{GO1}, with $w_1\in H^1(Q)$ satisfying \eqref{CGO11} for $j=1$, as well as the existence of  a solution $u_2\in L^2(0,T;H^1(\Omega))$ of $\Box u_2+q_2u_2=0$ in $Q$,  $\tau_{0,2}u_2=0$, taking the form \eqref{GOO1} with the term $w_2\in L^2(0,T;H^1(\Omega))$ satisfying \eqref{CGO01}. Repeating the arguments of the end of the proof of Theorem \ref{thm1} (see Subsection 4.4), we can  deduce the following orthogonality identity
\bel{thm2c}\lim_{\lambda\to+\infty}\int_Qqu_1u_2dxdt=0.\ee
Moreover, one can check that
$$\int_Qqu_1u_2dxdt=\int_{\R^{1+n}}q(t,x)e^{-i\xi\cdot(t,x)}dxdt+\int_{Q}Y_\lambda(t,x)dxdt,$$
with
$$Y_\lambda(t,x)=q[e^{-2\lambda t}e^{-i(t,x)\cdot\xi}+e^{-i(t,x)\cdot\xi}w_2+w_1+w_1w_2].$$
Combining \eqref{CGO11}, \eqref{CGO01} with the fact that
$$\int_Q|q(t,x)|\abs{e^{-2\lambda t}e^{-i(t,x)\cdot\xi}}dxdt\leq\norm{q}_{L^2(Q)}\left(\int_0^{+\infty}e^{-2\lambda t}dt\right)^{\frac{1}{2}}\leq \norm{q}_{L^2(Q)}\lambda^{-\frac{1}{2}},$$
we deduce that
$$\lim_{\lambda\to+\infty}\int_{Q}Y_\lambda(t,x)dxdt=0.$$
Combining this asymptotic property with \eqref{thm2c}, we can conclude in a similar way to Theorem \ref{thm1} that $q_1=q_2$. \qed

\section*{Acknowledgments}

 The work of the first author is supported by the NSFC grant (No. 11671028), NSAF grant (No. U1530401) and the 1000-Talent Program of Young Scientists in China. The second author would like to thank Pedro Caro for his remarks and  fruitful discussions about this problem. The second author would like to thank the Beijing Computational Science Research Center, where part of this article was written, for its kind hospitality

\end{document}